\documentclass[12pt]{article}

\usepackage{amsmath}
\usepackage{amsthm}
\usepackage{amsfonts}
\usepackage{amssymb}
\usepackage{graphicx}
\usepackage{latexsym}
\usepackage{mathrsfs}
\usepackage{epsfig}

\advance \topmargin by -\headheight
\advance \topmargin by -\headsep
\evensidemargin \oddsidemargin
\newtheorem{theorem}{Theorem}[section]
\newtheorem{thm}{Theorem}[section]

\newtheorem{lem}[theorem]{Lemma}

\newtheorem{prop}[theorem]{Proposition}

\theoremstyle{definition}

\newtheorem{defn}[theorem]{Definition}

\newtheorem{cor}[theorem]{Corollary}

\theoremstyle{remark}
\newtheorem{remark}[theorem]{Remark}

\numberwithin{equation}{section}

\newcommand\NN{\mathbb N}
\newcommand\RR{\mathbb R}
\newcommand\ZZ{\mathbb Z}

\newcommand\cB{\mathfrak{B}}
\newcommand\tcB{\widetilde{{\mathfrak{B}}}}

\newcommand\cG{\mathcal{G}}

\newcommand\cS{\mathfrak{S}}
\newcommand\tcS{\widetilde{{\mathfrak{S}}}}

\newcommand\cF{\mathfrak{F}}
\newcommand\tcF{\widetilde{\mathfrak{F}}}
\newcommand\cN{\mathcal{N}}

\newcommand\cX{\mathcal{X}}

\newcommand\abs[1]{\left|#1\right|}

\newcommand\set[1]{\left\{{#1}\right\}}

\def\cc{\curvearrowright}

\def\fh{{\mathfrak{v}}}

\def\sM{{\mathbb{M}}}

\def\Stab{{\rm St_\Gamma}}
\def\R{{\mathbb{R}}}
\def\Z{{\mathbb{Z}}}

\def\A{{\mathbb{A}}}

\def\M{{\mathbb{M}}}
\def\E{{\mathbb{E}}}
\def\L{{\mathbb{L}}}

\begin{document}
\title{von-Neumann and Birkhoff ergodic theorems for negatively curved  groups}

\author{Lewis Bowen\footnote{supported in part by NSF grant DMS-0968762, NSF CAREER Award DMS-0954606 and BSF grant 2008274}  ~and Amos Nevo\footnote{supported in part by ISF grant, and BSF grant 2008274}}







\maketitle
\begin{abstract}
We prove maximal inequalities for concentric ball and  spherical shell averages on  a general Gromov hyperbolic group,  in arbitrary probability preserving actions of the group. Under an additional condition, satisfied for example by all groups  acting isometrically and properly discontinuously  on $CAT(-1)$ spaces,  we prove a pointwise ergodic theorem with respect to a sequence of probability measures supported on concentric spherical  shells. 

%

\end{abstract}

\tableofcontents

\section{Introduction}
\subsection{Basic definitions}
Let $\Gamma$ be a countable group and $\{\zeta_r\}_{r>0}$ a family of probability measures on $\Gamma$. Given a pmp (probability-measure-preserving) action $\Gamma \cc (X,m)$, we can associate to  each $\zeta_r$  an operator $\pi_X(\zeta_r)$ on $L^p(X,m)$, acting  by:
$$\pi_X(\zeta_r)(f)(u) = \sum_{g\in \Gamma} \zeta_r(g) f(g^{-1}u).$$
We also consider the associated maximal function:
$$\M_\zeta[f](u):= \sup_{r>0}\pi_X( \zeta_r)(|f|).$$
We will usually suppress $\pi_X$ from the notation and write simply $\pi_X(\zeta_r)f=\zeta_r f$.  
Let us recall the following definitions : 
\begin{itemize}
\item $\{\zeta_r\}_{r>0}$ satisfies the {\em strong $L^p$ maximal inequality} if there is a constant $C_p>0$ such that $\|\M_\zeta[f]\|_p \le C_p \|f \|_p$ for every $f \in L^p(X,m)$;
\item $\{\zeta_r\}_{r>0}$ satisfies the {\em $L\log L$ maximal inequality} if there is a constant $C_1>0$ such that $\|\M_\zeta[f]\|_{L\log L} \le C_1 \|f \|_1$ for every $f \in L^1(X,m)$;
\item $\{\zeta_r\}_{r>0}$ is a {\em pointwise convergent family in $L^p$} if $\zeta_r(f)$ converges pointwise a.e. for every $f \in L^p(X,m)$;
\item $\{\zeta_r\}_{r>0}$ is a {\em pointwise ergodic family in $L^p$} if $\zeta_r(f)$ converges pointwise a.e. to $\E[f|\Gamma]$, the conditional expectation of $f$ on the sigma-algebra of $\Gamma$-invariant Borel sets (for every $f \in L^p(X,m))$.
\end{itemize}

The purpose of the present work is to establish maximal and pointwise ergodic theorems for natural geometric averages on word hyperbolic groups. Before describing the probability measures we will be interested in let us recall the following definitions. Given a proper left-invariant metric $d$ on $\Gamma$, the {\em Gromov product} of $x,y \in \Gamma$ relative to $z\in \Gamma$ is
$$(x|y)_z := \frac{1}{2}\Big( d(x,z) + d(y,z) - d(x,y)\Big).$$
The pair $(\Gamma, d)$ is a {\em hyperbolic group} if for some $\delta\ge 0$, 
\begin{eqnarray}\label{eqn:gromov}
(x|y)_w \ge \min\{ (x|z)_w, (y|z)_w\} - \delta, \quad  \forall x,y,w,z \in \Gamma.
\end{eqnarray}

Let $I \subset \RR$ be an interval. A map $\gamma:I \to \Gamma$ is a $(\lambda,c)$-quasi-geodesic if $\lambda^{-1}|i-j| - c \le d(\gamma(i),\gamma(j))  \le \lambda |i-j| + c$ for every $i,j$. We say that $(\Gamma,d)$ is {\em uniformly quasi-geodesic} if there exists a constant $c>0$ such that for every pair of elements $x,y \in \Gamma$ there exists a $(1,c)$-quasi-geodesic from $x$ to $y$. 

\subsection{Statement of main results}
Our first result concerns maximal inequalities for radial averages.  For $r>0$, let $\beta_r$ be probability measure on $\Gamma$ which is uniformly distributed on the ball $B(e,r)$ of radius $r$ centered at the identity. In other words, $\beta_r(g)=|B(e,r)|^{-1}$ if $g \in B(e,r)$ and $\beta_r(g)=0$ otherwise.  Similarly, for a fixed $a > 0$, we denote by $\sigma_{r,a}$ the uniform probability measure on the spherical shell $S_{r,a}(e)=\set{g\in \Gamma\,;\, r-a <  \abs{g}\le r+a}$. Finally we let $\mu_{r,a}=\frac{1}{r}\int_0^r \sigma_{s,a} ds$ be the uniform averages of the spherical shell measures. Our main maximal inequality is as follows. 

\begin{thm}\label{thm:maximal0}
Let $(\Gamma,d)$ be a non-elementary uniformly quasi-geodesic hyperbolic group. Then the family of ball averages $\{\beta_r\}_{r>0}$ satisfies the $L\log L$-maximal inequality and the strong $L^p$ maximal inequality for every $p>1$. The same holds for the families $\set{\sigma_{r,a}}_{r > 0}$ and $\set{\mu_{r,a}}_{r > 0}$, provided $a$ is larger than a fixed constant depending only on $\Gamma,d)$.  

\end{thm}

Our second result considers the case of word metrics $d_S$, where $S$ is a finite symmetric set of generators for $\Gamma$ and $d_S(g,e)=\abs{g}_S$ is the word length. 
We let $\sigma_n$ denote the uniform probability measure on the sphere $S_n(e)$  of radius $n$ and center $e$, and  $\mu_{n}=\frac{1}{n+1}\sum_{k=0}^n \sigma_{k}$ their uniform averages. We then have 

\begin{thm}\label{thm:word-metrics}  Let $\Gamma$ be a word-hyperbolic group, $S$ a symmetric set of generators.
Then $\mu_{n}$ satisfies the strong maximal inequality in $L^p$, $ 1 < p < \infty$, and in $L\log L$, and is a pointwise (and mean) convergent family in these spaces.
 \end{thm}

We remark that Theorem \ref{thm:word-metrics} extends  the  $L^2$-maximal inequality  for $\mu_{n}$ established in \cite{FN98},  improves on the mean and pointwise convergence  in $L^p$, $1 < p< \infty$ proved for $\mu_n$ in  \cite[Corollary 1]{BKK11} and \cite[Thm. 1]{BK12}, and generalizes the pointwise convergence for bounded functions established for $\mu_n$  under an additional assumption also  in \cite{PS11}. 

Let us turn now to the problem of pointwise (and mean) convergence of the balls and spherical shell averages on a hyperbolic group $(\Gamma,d)$. We will require an additional assumption on the group, and we will comment on its prevalence and on the optimality of the ergodic theorem it gives rise to after stating the theorem.

We will say that the {\em Gromov boundary coincides with the horofunction boundary} if for every point $\xi$ in the Gromov boundary of $(\Gamma,d)$, every sequence $\{x_i\}_{i=1}^\infty \subset \Gamma$ converging to $\xi$ and every $y \in \Gamma$, the limit
$$h_\xi(y):=\lim_{i\to\infty} d(x_i,y) - d(x_i,e)$$
exists and depends only on $\xi, y$. Our main pointwise ergodic theorem is:

\begin{thm}\label{thm:main}
Let $(\Gamma,d)$ be a non-elementary uniformly quasi-geodesic  hyperbolic group whose Gromov boundary coincides with its horofunction boundary. Then for each $a>0$ larger than a fixed constant depending only on $(\Gamma,d)$,  there exists a family $\{\kappa_r\}_{r>0}$ of probability measures on $\Gamma$ such that
\begin{enumerate}
\item  each $\kappa_r$ is supported on the spherical shell $S_{r,a}(e)=\{g\in \Gamma:~ d(e,g) \in [r-a,r+a]\}$,
\item $\{\kappa_r\}_{r>0}$ is a pointwise (and mean) ergodic family in $L^p$ for every $p>1$ and in $L\log L$. 
\end{enumerate}
\end{thm}




\subsection{Comments on ergodic theorems for radial averages}

As to the hypotheses of Theorem \ref{thm:main}, let us note that the coincidence of the horofunction boundary and the Gromov boundary is a common phenomenon, but not a universal one. Let us give two examples where it is satisfied, and one where it may fail.  
\begin{enumerate}
\item {\bf CAT(-1) spaces}. Suppose $\Gamma$ acts properly discontinuously by isometries on a CAT(-1) space $(X,d_X)$. For $x \in X$ (with trivial stability group in $\Gamma$) define the metric $d$ on $\Gamma$ by $d(g,g'):=d_X(gx,g'x)$. For a proof of the coincidence of the boundaries for this metric see \cite[Chapter II.8, Theorem 8.13 and Chapter III.H]{BH99}.
\item {\bf The Green metric}. If $\zeta$ is a finitely supported symmetric probability measure on $\Gamma$ whose support generates $\Gamma$ then the {\em Green metric} induced by $\zeta$ is defined as follows. Let $X_1,X_2,\ldots$ be a sequence of independent, identically distributed random variables each with law $\zeta$. Let $Z_0=e$ and $Z_n=X_1\cdots X_n$ (for $n\ge 1$) be the random walk on $\Gamma$ induced by $\zeta$. For $g,g' \in \Gamma$, let $d_\zeta(g,g') = -\log( p(g,g'))$ where $p(g,g')$ is the probability that $gZ_n=g'$ for some $n\ge 0$. This is the Green metric induced by $\zeta$. The informative paper \cite{Bj10} explains why the horofunction boundary of $(\Gamma, d_\zeta)$ coincides with its Gromov boundary. This is based on work of Ancona \cite{An87, An88, An90} showing that the Martin boundary of $(\Gamma,d_\zeta)$ equals its Gromov boundary. In \cite{BHM11} it is shown that the Green metric is uniformly quasi-geodesic and Gromov hyperbolic if $\Gamma$ is word hyperbolic. It is obviously proper and left-invariant. It follows that every non-elementary finitely generated word hyperbolic group has a metric $d$ satisfying the hypotheses of Theorem \ref{thm:main}.
\item {\bf Word metrics}. By \cite{CP01}, the horofunction boundary of $\Gamma$ with an arbitrary word metric admits a canonical finite-to-1 $\Gamma$-equivariant map onto the Gromov boundary. However, this map need not be a homeomorphism. 
For example, if $\Gamma=\Gamma_0 \times F$ where $\Gamma_0$ is word hyperbolic and $F$ is a finite nontrivial group then for any word metric on $\Gamma$ that is induced by a generating set of the form $\{(g, e_F):~g\in S\} \cup \{(e_{\Gamma_0},f):~f\in F\}$ where $S$ is a generating set for $\Gamma_0$,  the horofunction boundary does not coincide with the Gromov boundary.
\end{enumerate}

Let us now comment briefly on the optimality of Theorem   \ref{thm:main} in the context of ergodic theorems for ball averages on hyperbolic groups. To begin with, recall the well-known fact  that the ball averages  on the free group, defined with respect to a free set of generators, actually  fail to converge, in general. 
Indeed, whenever $L^2_0(X,m)$ contains an eigenfunction $\phi$ with $\pi_X(\sigma_n)\phi=(-1)^n \phi$, clearly the sequences $\pi_X(\sigma_n )\phi$ and $\pi_X(\beta_n) \phi$ do not converge, and the same holds true for spherical shells. Thus even when the horofunction boundary and the Gromov boundary coincide, as in the case of the free group acting on its Cayley tree, the ball averages (and the spherical shell averages)  do not satisfy the ergodic theorem.  This  is not an isolated fact, and for example in the groups $\ZZ_p\ast \ZZ_q$ where $p \neq q$,  the ergodic theorem likewise fails for ball (and spherical shell) averages, see the discussion following \cite[Thm. 10.7]{Ne05}.  

 Let us now point out however that in the cases just mentioned there exists a  sequence of probability measures supported on the spherical shells $S_{n,1}(e)$ which is indeed  a pointwise ergodic sequence. It is given by $\sigma_{n}^\prime=\frac12\left(\sigma_n+\sigma_{n+1}\right) $ in these examples, see \cite{Ne94a}. By the previous comment, the probability measures in question supported on the shell $S_{n,1}(e)$,  are necessarily non-uniform in the examples mentioned. Thus Theorem \ref{thm:main} gives essentially the optimal result in this  case, namely pointwise convergence for probability measures supported on spherical shells.   

Finally, let us point out that in some cases, the ball averages associated with a hyperbolic metric $d$ on $\Gamma$ do in fact form a pointwise ergodic sequence in $L^p$, $ 1 < p <\infty$. Indeed, let $G$ be connected almost simple real Lie group of real rank one, and $\Gamma$ a uniform lattice in $G$. Let $\cal{S}$ denote the symmetric space associated with 
$G$, and fix a point $p\in \cal{S}$ whose stabilizer in $\Gamma$ is trivial. Let $D_{\cal{S}}$ be the $G$-invariant Riemannian metric on symmetric space, and define $d_{\cal{S}}(g,h)=D_{\cal{S}}(gp,hp)$ for $g,h\in \Gamma$. Then  $d_{\cal{S}}$ is a hyperbolic metric on $\Gamma$, and the associated ball averages $\beta_r$ are a pointwise (and mean) ergodic family in $L^p$, $1 < p < \infty$. This fact appears in \cite{GN10}, and its proof depends on detailed information regarding  the unitary representation  theory of the simple Lie group $G$.

\subsection{A brief sketch}


Let $\partial\Gamma$ denote the Gromov boundary of $(\Gamma,d)$. Via the Patterson-Sullivan construction, there is a quasi-conformal probability measure $\nu$ on $\partial\Gamma$. So there are constants $C,\fh>0$ such that
$$C^{-1}\exp(-\fh h_\xi(g^{-1})) \le \frac{d\nu \circ g}{d\nu}(\xi) \le C \exp(-\fh h_\xi(g^{-1}))$$
for every $g\in \Gamma$ and a.e. $\xi \in \partial \Gamma$ where
$$h_\xi(g^{-1}):= \inf \liminf_{n\to\infty} d(x_i,g^{-1}) - d(z_i,e)$$
where the infimum is over all sequences $\{x_i\} \subset \Gamma$ converging to $\xi$. 

The {\em type} of the action $\Gamma \cc (\partial \Gamma,\nu)$ encodes the essential range of the Radon-Nikodym derivative. In \cite{Bo12}, it is shown that this type is $III_\lambda$ for some $\lambda \in (0,1]$. If $\lambda \in (0,1)$ then we set
$$R_\lambda(g,\xi)= - \log_\lambda \left( \frac{d \nu \circ g}{d\nu}(\xi) \right).$$ Using standard results, it can be shown that then we can choose $\nu$ so that $R_\lambda(g,\xi) \in \ZZ$ for every $g$ and a.e. $\xi$. When $\lambda=1$, we set $R_1(g,\xi) = +\log \left( \frac{d \nu \circ g}{d\nu}(\xi) \right)$. 

In order to handle each case uniformly, we set $\L=\RR$ if $\lambda=1$ and $\L=\ZZ$ if $\lambda \in (0,1)$. Then we let $\Gamma$ act on $\partial\Gamma \times \L$ by
$$g(\xi,t) = (g \xi, t-R_\lambda(g,\xi)).$$
This action preserves the measure $\nu \times \theta_\lambda$ where $\theta_1$ is the measure on $\RR$ satisfying $d\theta_1(t) = \exp( t) dt$ and, for $\lambda \in (0,1)$, $\theta_\lambda$ is the measure on $\ZZ$ satisfying $\theta_\lambda(\{n\}) = \lambda^{-n}$. 

Given $a,b \in \L$, let $[a,b]_\L \subset \L$ denote the interval $\{a,a+1,\ldots, b\}$ if $\L=\Z$ and $[a,b] \subset \RR$ if $\L=\RR$. Similar considerations apply to open intervals and half-open intervals.

For any real numbers $r, T>0$, and $(\xi,t) \in \partial\Gamma\times [0,T)_L$, let 
$$\Gamma_r(\xi,t) = \{g \in \Gamma:~ d(g,e) - h_\xi(g)-t \le r, ~g^{-1}(\xi,t) \in \partial \Gamma \times [0,T)_\L \}$$
and
$$\cB^{}_r(\xi,t) := \{ g^{-1} (\xi,t):~ g\in \Gamma_r(\xi,t) \}.$$
$\Gamma_r(\xi,t)$ is approximately equal to the intersection of the ball of radius $r$ centered at the identity with the horoshell $\{g \in \Gamma:~ -t\le h_\xi(g) \le T-t\}$.  Of course, $\Gamma_r$ and $\cB_r$ depend on $T$, but we leave this dependence implicit. 


Our first main technical result is that if $T$ is sufficiently large then $\cB$ is {\em regular}: there exists a constant $C>0$ such that for every $r>0$ and a.e. $(\xi,t) \in \partial\Gamma \times [0,T)_\L$,
$$| \cup_{s\le r} \cB^{-1}_s \cB_r (\xi,t)| \le C |\cB_r(\xi,t)|.$$
Theorem \ref{thm:maximal0} now follows from an extension of the general results in \cite{BN2, BN3}. 
The idea is that we can use the regularity of the sets $\cB_r$ and prove a maximal inequality for them, and thus for the equivalence relation on $\partial\Gamma\times [0,T)_\L$ given by the intersection of the $\Gamma$-orbits with this subset. We can then average this maximal inequality over $\partial \Gamma$ to obtain a maximal inequality for the resulting family of probability measures on $\Gamma$. By geometric arguments, the family we obtain is sufficiently close to being uniform averages on balls that this implies a maximal inequality for the uniform averages on balls.

The results of \cite{BN2,BN3} do not directly apply because the action of $\Gamma$ on its boundary might not be essentially free. However, this action has uniformly bounded stabilizers. Using this hypotheses we generalize the needed theorems of \cite{BN2,BN3} in \S \ref{sec:er}-\ref{sec:general}.

Next we let $\cS^{}_a=\{\cS^{}_{r,a}\}_{r>0}$ be the family of subset functions on $\partial\Gamma\times [0,T)_\L$ defined by 
$$\cS^{}_{r,a}(\xi,t) := \cB_r(\xi,t) \setminus \cB_{r-a}(\xi,t)$$
and observe that $\cS_a$ is also regular if $a,T>0$ are sufficiently large. 

Our second main technical result is that $\cS_a$ is {\em asymptotically invariant} (modulo a minor technical issue) assuming that the horofunction boundary coincides with the Gromov boundary. To explain, we let $E$ denote the equivalence relation on $\partial\Gamma \times [0,T)_\L$ given by $(\xi,t)E(\xi',t')$ if there exists $g \in \Gamma$ such that $ (\xi,t)=g(\xi',t')$. The {\em full group} of $E$ is the group of all (equivalence classes of) Borel automorphisms on $\partial\Gamma \times [0,1]$ with graph contained in $E$, where two such automorphisms are equivalent if they agree almost everywhere. It is denoted by $[E]$. A subset $\Phi \subset [E]$ {\em generates $E$} if for a.e. $(\xi,t) \in \partial\Gamma \times [0,T)_\L$ and every $(\xi',t')$ with $(\xi,t)E(\xi',t')$ there is an element $\phi \in \langle \Phi \rangle$ (the subgroup generated by $\Phi$) such that $\phi(\xi,t)=(\xi',t')$. Finally, $\cS_a$ being asymptotically invariant means that there exists a countable generating set $\Phi \subset [E]$ such that 
$$\lim_{r \to \infty} \frac{ |\cS_{r,a}(\xi,t) \vartriangle \phi(\cS_{r,a}(\xi,t))|}{|\cS_{r,a}(\xi,t)|} = 0$$
for every $\phi \in \Phi$ and a.e. $(\xi,t)$. 

Using asymptotic invariance, it now follows from general results of \cite{BN2,BN3} (as generalized in \S \ref{sec:er}-\ref{sec:general}) that there is a pointwise {\em convergent} family $\{\kappa_r\}_{r>0}$ of probability measures on $\Gamma$ and a constant $a>0$ so that each $\kappa_r$ is supported on the annulus $\{g\in \Gamma:~ d(e,g) \in [r-a,r+a]\}$. However, this is not the end of the proof because at this stage we only know that for any pmp action $\Gamma \cc (X,m)$ and any $f \in L^p(X,m)$ that $\kappa_r(f)$ converges almost everywhere. We have not yet identified what it converges to!

The issue is that even if $\Gamma \cc (X,m)$ is ergodic, it does not necessarily follow that the product action $\Gamma \cc (X \times \partial\Gamma\times \L, m \times \nu \times \theta_\lambda)$ is ergodic. To resolve this, first we show that $\Gamma \cc (\partial\Gamma,\nu)$ is weakly mixing (so $\Gamma \cc (X\times \partial\Gamma, m \times \nu)$ is ergodic). This uses the fact that Poisson boundary actions are weakly mixing \cite{AL05} and that the action is equivalent to a Poisson boundary action \cite{CM07}. From \cite{Bo12}, it follows that $\Gamma \cc (\partial\Gamma,\nu)$ has type $III_\rho$ and stable type $III_\tau$ for some $\rho, \tau \in (0,1]$. From this it follows that the natural cocycle $\alpha:\tilde{E} \to \Gamma$ (where $\tilde{E}$ is the equivalence relation on $X \times \partial \Gamma \times [0,T)_\L$) is weakly mixing relative to a certain compact group. This is ultimately what is needed to invoke Theorems \ref{thm:III_1}, \ref{thm:III_lambda} (which generalize \cite{BN2,BN3}) and thereby complete the proof.





{\bf Organization of the paper}. \S \ref{sec:er} discusses maximal and ergodic theorems for measured equivalence relations. This is used in \S \ref{sec:general} to obtain some general ergodic theorems which will be used to prove the main results. Then \S \ref{sec:Gromov} reviews Gromov hyperbolic groups and sets some notation. \S \ref{sec:doubling} establishes the regularity of the averaging sets and proves Theorem \ref{thm:maximal0} and Theorem  \ref{thm:word-metrics}. In \S \ref{sec:invariance}, we prove asymptotic invariance of the averaging sets. The last section \S \ref{sec:main} uses asymptotic invariance to complete the proof of Theorem \ref{thm:main}.

{\bf Convention}. We have not attempted to produce explicit estimates for the constants appearing in the statements of the main results, and throughout the paper we use the "variable constant convention", namely in different occurrences of a constant  (even within the same argument) the values it assumes may be different.


\section{Equivalence relations and ergodic sequences}\label{sec:er}

\subsection{An ergodic theorem for equivalence relations}
The purpose of this section is to review and generalize the main theorems of \cite{BN2,BN3} so that we can later apply them to Gromov hyperbolic groups. 
To this end, let $(B,\nu)$ be a standard probability space and $E \subset B\times B$ a discrete measurable equivalence relation. Let $[E]$ denotes the {\em full group} of $E$, namely  the group of all measurable automorphisms of $B$ with graph contained in $E$ (discarding a null set).   We assume that $\nu$ is $E$-invariant, namely that $\phi_*\nu=\nu$ for every $\phi \in [E]$.

We will obtain  theorems for $E$ first and then push them forward via a cocycle $E \to \Gamma$. We begin by discussing ergodic theorems for $E$, which utilize finite subset functions, defined next.

Let $2^B$ be the set of all subsets of $B$. A (finite)  {\em subset function} for $E$ is a map $\cF:B \to 2^B$ such that $\cF(\xi)$ is finite and $\cF(\xi) \subset [\xi]_E=\{\eta \in B:~(\xi,\eta)\in E\}$ for almost every $\xi$.  $\cF$ is called measurable if the set $\{(r,\xi,\eta) \in \R \times B \times B:~ \eta\in \cF_r(\xi)\}$ is measurable. 

We let $\cF^{-1}$ be the subset function defined by
$$\cF^{-1}(\eta) = \{\xi \in B:~ \eta \in \cF(\xi)\}.$$
If $\cF,\cG$ are two subset functions on $B$ then their product is defined by
$$\cF\cG(\xi) = \bigcup_{\eta \in \cG(\xi)} \cF(\eta).$$
 The composition,  union, intersection and relative complement of two subsets functions (defined in the obvious way)  constitute subset functions in their own right, namely they are measurable and finite for almost every $\xi\in B$.

We will be interested in averaging over such subset functions. First, let $\alpha:E \to \textrm{Aut}(X,m)$ be a measurable cocycle into the automorphism group of a standard Borel space. In particular, we require $\alpha(\xi,\eta)\alpha(\eta,\xi^\prime)=\alpha(\xi,\xi^\prime)$ (for a.e. $\xi\in B$ and every $\eta,\xi^\prime \in [\xi]_E$). Let $E_\alpha$ be the induced equivalence relation on $B\times X$. So $(\xi,u)E_\alpha(\xi',u')$ if and only if $\xi E\xi'$ and $\alpha(\xi',\xi)u=u'$.

Let $f \in L^p(B\times X, \nu \times m)$ for some $p$. Given a family of $\cF=\{\cF_r\}_{r>0}$ of subset functions for $E$, we consider the averages
$$\A[f|\cF_r](\xi,u):= |\cF_r(\xi)|^{-1} \sum_{(\xi') \in \cF_r(\xi,u)} f(\xi', \alpha(\xi',\xi)^{-1}u')$$
and the maximal function
$$\M[f|\cF](\xi,u) := \sup_{r>0} \A[|f| | \cF_r](\xi,u).$$
Our assumption that the equivalence classes $[\xi]_E$ are almost always countable, and the subset  $\cF(\xi)\subset [\xi]_E$ almost always finite implies that the maximal function is measurable. We say : 
\begin{itemize}
\item $\cF$ satisfies the {\em weak $(1,1)$-type maximal inequality} if there is a constant $C>0$ such that for all $t>0$ and all $f\in L^1(B\times X)$,
$$\nu \times m(\{ (\xi,u):~ \M[f |\cF](\xi,u)>t\}) \le C\frac{ \|f \|_1}{t};$$
\item $\cF$ satisfies the {\em strong $p$-type maximal inequality} if there is a constant $C_p>0$ such that $\| \M[f|\cF] \|_p \le C_p \|f \|_p$ for every $f \in L^p(B\times X)$;
\item $\cF$ is a {\em pointwise ergodic family in $L^p$} if for every $f\in L^p(B\times X)$, $\A[f|\cF_r]$ converges pointwise a.e. to $\E[f| E_\alpha]$, the conditional expectation of $f$ on the $\sigma$-algebra of $E_\alpha$-saturated sets.
\end{itemize}

Next we provide conditions which imply the conditions above.

A family $\cF=\{\cF_r\}_{r>0}$ of subset functions on $B$ is {\em regular} if there exists a constant $C>0$ (called a {\em regularity constant}) such that for every $r>0$ and a.e. $\xi\in B$,
$$\left| \bigcup_{s\le r} \cF^{-1}_s\cF^{}_r(\xi) \right| \le C |\cF^{}_r(\xi)|.$$

A subset $\Phi \subset [E]$ {\em generates $E$} if for a.e. $\xi\in B$ and every $\eta \in [\xi]_E$, there is $\phi \in \langle \Phi \rangle$ (the subgroup generated by $\Phi$) such that $\phi(\xi)=\eta$. The family $\cF$ is {\em asymptotically invariant} if there exists a countable set $\Phi \subset [E]$ which generates $E$ such that 
$$\lim_{r\to\infty} \frac{ |\cF_r(\xi) \vartriangle \phi(\cF_r(\xi))| }{|\cF_r(\xi)|} = 0\quad \textrm{ for a.e. $\xi$}, \forall \phi \in \Phi.$$
We now recall the following, which is part of \cite[Theorems 2.4-2.6]{BN2}.
\begin{thm}\label{thm:pointwise-eq}
If $\cF$ is regular then it satisfies the weak $(1,1)$-type maximal inequality and the strong $p$-type maximal inequality for all $p>1$. If, in addition, it is asymptotically invariant then it is a pointwise ergodic family in $L^p$ (for every $p\ge 1$).
\end{thm}

Most of the effort in this paper goes towards showing that certain subset functions on equivalence relations obtained from the action of a hyperbolic group on its boundary are both regular and asymptotically invariant. Next we explain how these results imply pointwise ergodic theorems for  $\Gamma$.  We also need to generalize previous results because the action of $\Gamma$ on its boundary need not be essentially free.

\subsection{From equivalence relations to ergodic sequences }\label{sec:general}

We begin by recalling the definition of the Maharam extension of a general non-singular action, and the method of deriving ergodic theorems from it.  

\begin{defn}\label{defn:maharam}
The {\em Maharam extension} of a measure-class-preserving action $\Gamma \cc (B,\nu)$ is the action $\Gamma \cc (B \times \RR, \nu \times \theta)$ defined by
$$\gamma(\xi,t):=(\gamma \xi, t - R(g,\xi)), \quad R(g,\xi) := \log\left( \frac{d\nu \circ g}{d\nu}(\xi) \right)$$
and $\theta$ is the measure on $\RR$ satisfying $d\theta(t)=e^tdt$. This action is measure-preserving.

If $\Gamma \cc (B,\nu)$ is ergodic then we say $\Gamma \cc (B,\nu)$ has {\em type $III_1$} if the Maharam extension is also ergodic.
\end{defn}

\begin{defn}\label{defn:ubs}
A measure-class preserving action $\Gamma \cc (B,\nu)$ has {\em uniformly bounded stabilizers} if there is a  constant $C>0$ such that for a.e. $\xi \in B$, $|\Stab(\xi)|\le C$ where $\Stab(\xi)=\{g\in \Gamma:~g\xi=\xi\}$.
\end{defn}

\begin{thm}[The $III_1$ case]\label{thm:III_1}
Let $\Gamma$ be a countable group and $\Gamma \cc (B,\nu)$ a measure-class preserving action on a standard probability space with uniformly bounded stabilizers. Let $T>0$, $\cF=\{\cF_r\}_{r>0}$ be a measurable family of set functions for the equivalence relation $E$ on $B\times [0,T]$ (induced from the Maharam extension $\Gamma \cc B\times \R$ as above) and let $\psi\in L^q(B,\nu)$ be a probability density (so $\psi \ge 0$, $\int \psi~d\nu=1$). 

Define probability measures $\zeta_r$ on $\Gamma$ by
$$\zeta_r(g) = T^{-1} \int_0^T \int |\{w\in \Gamma:~w(\xi,t)\in \cF_r(\xi,t)\}|^{-1}1_{\cF_r(\xi,t)}(g^{-1}(\xi,t))\psi(\xi)~d\nu(\xi) dt.$$

Let $p>1$ be such that $\frac{1}{p}+\frac{1}{q}=1$. 
\begin{itemize}
\item If $\cF$ is regular then $\{\zeta_r\}$ satisfies the strong $L^p$ maximal inequality. If $\psi \in L^\infty(B,\nu)$ then $\{\zeta_r\}$ satisfies the $L\log L$ maximal inequality.
\item If $\cF$ is regular and asymptotically invariant then $\{\zeta_r\}$ is a pointwise convergent family in $L^p$ (and if $\psi \in L^\infty(B,\nu)$ then it is pointwise convergent in $L\log L$).
\item If $\cF$ is regular, asymptotically invariant, $\Gamma \cc (B,\nu)$ is weakly mixing, type $III_1$ and stable type $III_\lambda$ (where either $\lambda=1$ or $T$ is a positive integer multiple of $-\log(\lambda)>0$) then $\{\zeta_r\}$ is a pointwise ergodic family in $L^p$ (and if $\psi \in L^\infty(B,\nu)$ then it is pointwise ergodic in $L\log L$).
\end{itemize}
\end{thm}

We refer to \cite{BN3} for background on type and stable type. We say that $\Gamma \cc (B,\nu)$ is {\em weakly mixing} if for any ergodic pmp action $\Gamma \cc (X,m)$, the product action $\Gamma \cc (B\times X,\nu\times m)$ is ergodic.

\begin{proof}
To begin, let us assume that $\Gamma \cc (B,\nu)$ is essentially free. We will show that this result follows from \cite[Theorems 3.1 and 5.1]{BN3}. By Theorem \ref{thm:pointwise-eq}, if $\cF$ is regular then it satisfies the weak $(1,1)$-type maximal inequality and if it is both regular and asymptotically invariant then it is poinwise ergodic in $L^p$ (for every $p\ge 1$).

Let $\alpha: E \to \Gamma$ be the cocycle $\alpha(\eta,\xi)=\gamma$ if $\gamma \xi =\eta$. For $(\eta,\xi)\in E$, let $\omega_r(\eta,\xi) = |\cF_r(\xi)|^{-1}$ if $\eta \in \cF_r(\xi)$ and $\omega_r(\eta,\xi)=0$ otherwise. If $\cF$ is regular then $\Omega=\{\omega_r\}$ satisfies the weak $(1,1)$-type maximal inequality and the strong $L^p$ maximal inequality (in the sense of \cite[\S 2.1]{BN3}). If $\cF$ is regular and asymptotically invariant, then $\Omega$ is a pointwise ergodic family in $L^p$ (for every $p\ge 1$). 

 Let $K=\R/T\Z$ act on $B \times [0,T]$ by $k(\xi,t)=(\xi,t+k)$ where $t+k$ is taken modulo $T$. By \cite[Theorem 5.1]{BN3}, if $\Gamma \cc (B,\nu)$ is weakly mixing, type $III_1$ and stable type $III_\lambda$ (where either $\lambda=1$ or $0<T=-\log(\lambda)<\infty$)  then $\alpha$ is weakly mixing relative to the action of $K$. The result now follows from \cite[Theorem 3.1]{BN3}.

Let us suppose now that $\Gamma \cc (B,\nu)$ is not necessarily essentially free but does have uniformly bounded stabilizers. Let $\Gamma \cc (Y,p)$ be a nontrivial Bernoulli shift action. This action is essentially free, pmp and strongly mixing. It therefore enjoys the following multiplier property: if $\Gamma \cc (X,m)$ is any properly ergodic action then the product action $\Gamma \cc (X\times Y, m \times p)$ is also ergodic. In particular, it is not necessary for $\Gamma \cc (X,m)$ to be probability-measure-preserving. 

Observe that the product action $\Gamma \cc (B \times Y, \nu\times p)$ is essentially free. Let $\Gamma \cc B\times Y \times \RR$ denote the Maharam extension of $\Gamma \cc B\times Y$ and $\tilde{E}$ the induced equivalence relation on $B\times Y\times [0,T]$. Define the subset function $\tcF_r$ on $B\times Y \times [0,T]$ by $\tcF_r(\xi,y,t) = \{g(\xi,y,t):~g\in \Gamma, g(\xi,t) \in \cF_r(\xi,t)\}$. Note that 
\begin{eqnarray}\label{eqn:zeta}
\zeta_r(g)=\frac{1}{T} \iiint_0^T |\tcF_r(\xi,y,t)|^{-1}1_{\tcF_r(\xi,y,t)}(g^{-1}(\xi,y,t))\psi(\xi)~dtd\nu(\xi)dp(y).
\end{eqnarray}

Because $\Gamma \cc (B,\nu_B)$ has uniformly bounded stabilizers, there is a constant $C>0$ such that
$$|\cF_r(\xi,t)| \le |\tcF_r(\xi,y,t)| \le C |\cF_r(\xi,t)|~\textrm{ for a.e. $(\xi,y,t)$.}$$
If $\cF$ is regular, this implies $\tcF:=\{\tcF_r\}_{r>0}$ is also regular. Therefore, the essentially free case implies that $\{\zeta_r\}$ satisfies the strong $L^p$-maximal inequality and, if $q=\infty$, the $L\log L$-type maximal inequality.

Let us suppose now that $\cF$ is asymptotically invariant. We will show that $\tcF$ is asymptotically invariant. There exists a countable set $\Phi \subset [E]$ such that $\Phi$ generates $E$ and 
$$\lim_{r\to\infty} \frac{ |\phi(\cF_r(\xi,t)) \vartriangle \cF_r(\xi,t)|}{|\cF_r(\xi,t)|} = 0$$
for a.e. $(\xi,t)$ and every $\phi \in \Phi$.

Let $J:B \times Y \times [0,T] \to [0,1]$ be a Borel isomorphism and choose $L:B \times Y\times [0,T] \to  \{1,2,\ldots,C\}$ to satisfy: for a.e. $(\xi,y,t)$ 
\begin{itemize}
\item if $g \in \Stab(\xi,t)$ and $J(\xi,y,t) < J(g(\xi,y,t))$ then $L(\xi,y,t) < L(g(\xi,y,t))$;
\item $\max \{ L(g(\xi,h,t)):~ g\in \Stab(\xi,t)\} = |\Stab(\xi,t)|$.
\end{itemize}
For each $\phi \in \Phi$ and $n\in \Z$ define $\tilde{\phi}_n \in [\tilde{E}]$ by $\tilde{\phi}_n(\xi,y,t)=(\xi',y',t')$ where $\phi(\xi,t)=(\xi',t')$, $(\xi,y,t)\tilde{E}(\xi',y',t')$ and $L(\xi,y,t)\equiv L(\xi',y',t')+n \mod |\Stab(\xi,t)|$. This is well-defined almost everywhere. Observe that for any $\phi \in \Phi$ and a.e. $(\xi,y,t) \in B\times Y\times [0,T]$ if $g\in \Gamma$ is such that $\phi(\xi,t)=g(\xi,t)$ then there exists $i \in \Z$ so that $\tilde{\phi}_i(\xi,y,t)=g(\xi,y,t)$. Therefore $\tilde{\Phi}:=\{\tilde{\phi}_n:~ \phi \in \Phi, n \in \Z\}$ is generating. 

This construction implies that for any $\phi \in \Phi, n \in \Z$ and a.e. $(\xi,y,t) \in B \times Y\times [0,T]$, $|\tcF_r(\xi,y,t)| \le C |\cF_r(\xi,t)|$ and 
$$|\tilde{\phi}_n(\tcF_r(\xi,y,t)) \vartriangle \tcF_r(\xi,y,t)| \le C|\phi(\cF_{r}(\xi,t)) \vartriangle \cF_{r}(\xi,t)|.$$
So
$$ \lim_{r\to\infty} \frac{ |\tilde{\phi}_n(\tcF_r(\xi,y,t)) \vartriangle \tcF_r(\xi,y,t)| }{ |\tcF_r(\xi,y,t) | } = 0$$
which implies $\tcF$ is asymptotically invariant. So (\ref{eqn:zeta}) and the essentially free case imply $\{\zeta_r\}$ is a pointwise convergent family in $L^p$ (and in $L \log L$ if $q=+\infty$). 

Next let us assume that $\Gamma \cc B$ is weakly mixing, type $III_1$ and stable type $III_\lambda$. Because $\Gamma \cc Y$ is weakly mixing and pmp, it follows immediately that the product action $\Gamma \cc B\times Y$ is weakly mixing and stable type $III_\lambda$. Because $\Gamma \cc B\times \RR$ is ergodic, the action $\Gamma \cc (B\times \RR) \times Y$ is also ergodic. But this is isomorphic to the Maharam extension of $\Gamma \cc B\times Y$. Therefore, $\Gamma \cc B\times Y$ has type $III_1$. The conclusion now follows from the essentially free case and (\ref{eqn:zeta}).
\end{proof}

\begin{defn}
Suppose $\Gamma \cc (B,\nu)$ is a measure-class-preserving action, $\lambda \in (0,1)$ and the Radon-Nikodym derivatives satisfy
$$R_\lambda(g,\xi):=-\log_\lambda\left( \frac{d\nu \circ g}{d\nu}(\xi) \right) \in \ZZ$$
for a.e. $\xi \in B$ and $g\in \Gamma$. In this case, we consider the {\em discrete Maharam extension} which is the action $\Gamma \cc (B \times \ZZ, \nu \times \theta_\lambda)$ defined by
$$\gamma(\xi,t):=(\gamma \xi, t - R_\lambda(g,\xi))$$
where $\theta_\lambda$ is the measure on $\ZZ$ satisfying $\theta_\lambda(\{n\}) = \lambda^{-n}$. This action is measure-preserving.

If, in addition, $\Gamma \cc (B,\nu)$ is ergodic then $\Gamma \cc (B,\nu)$ has {\em type $III_\lambda$} if the discrete Maharam extension is also ergodic. (Type $III_\lambda$ is also well-defined if the Radon-Nikodym derivatives do not satisfy the above condition: see \cite{KW91} or \cite{BN3} for background on type).
\end{defn}

\begin{thm}[The $III_\lambda$ case]\label{thm:III_lambda}
Let $\Gamma$ be a countable group and $\Gamma \cc (B,\nu)$ a measure-class preserving action on a standard probability space with uniformly bounded stabilizers and Radon-Nikodym derivatives which satisfy
$$R_\lambda(g,\xi):=-\log_\lambda\left( \frac{d\nu \circ g}{d\nu}(\xi) \right) \in \ZZ$$
for a.e. $\xi \in B$ and $g\in \Gamma$. 

Let $\cF=\{\cF_r\}_{r>0}$ be a measurable family of set functions for the equivalence relation $E$ on $B\times \{0,1,\ldots, N-1\}$ (induced from the discrete Maharam extension $\Gamma \cc B\times \ZZ$ as above) and let $\psi\in L^q(B,\nu)$ be a probability density (so $\psi \ge 0$, $\int \psi~d\nu=1$). 

Define probability measures $\zeta_r$ on $\Gamma$ by
$$\zeta_r(g) = N^{-1} \sum_{t =0}^{N-1} \int |\{w\in \Gamma:~w(\xi,t)\in \cF_r(\xi,t)\}|^{-1}1_{\cF_r(\xi,t)}(g^{-1}(\xi,t))\psi(\xi)~d\nu(\xi).$$

Let $p>1$ be such that $\frac{1}{p}+\frac{1}{q}=1$. 
\begin{itemize}
\item If $\cF$ is regular then $\{\zeta_r\}$ satisfies the strong $L^p$ maximal inequality. If $\psi \in L^\infty(B,\nu)$ then $\{\zeta_r\}$ satisfies the $L\log L$ maximal inequality.
\item If $\cF$ is regular and asymptotically invariant then $\{\zeta_r\}$ is a pointwise convergent family in $L^p$ (and if $\psi \in L^\infty(B,\nu)$ then it is pointwise convergent in $L\log L$).
\item If $\cF$ is regular, asymptotically invariant, $\Gamma \cc (B,\nu)$ is weakly mixing, type $III_\lambda$ and stable type $III_\tau$ for $\tau=\lambda^{m}$ (some $m\in \NN$ such that $(N/m) \in \ZZ$) then $\{\zeta_r\}$ is a pointwise ergodic family in $L^p$ (and if $\psi \in L^\infty(B,\nu)$ then it is pointwise ergodic in $L\log L$).
\end{itemize}
\end{thm}

\begin{proof}
The essentially free case follows from \cite[Theorems 3.1 and 5.2]{BN3}. The rest of the proof is analogous to the proof of Theorem \ref{thm:III_1} so we leave it to the reader. 
\end{proof}

\section{Gromov hyperbolic spaces}\label{sec:Gromov}
We now turn to establish some properties of  Gromov hyperbolic spaces, the Gromov boundary and the horofunction boundary. These results will be applied to the case where  $(\Gamma,d)$ is a nonelementary uniformly quasi-geodesic hyperbolic group.

\subsection{The Gromov boundary}

Let $(\cX,d_\cX)$ be a $\delta$-hyperbolic space. A sequence $\{x_i\}_{i=1}^\infty$ in $\cX$ is a {\em Gromov sequence} if
$$\lim_{i,j \to \infty} (x_i|x_j)_z = +\infty$$
for some (and hence, any) basepoint $z\in \cX$. Two Gromov sequences $\{x_i\}_{i=1}^\infty$, $\{y_i\}_{i=1}^\infty $ are {\em equivalent} if $\lim_{i\to\infty} (x_i|y_i)_z = +\infty$ with respect to some (and hence any) basepoint $z$. It is an exercise to show that this defines an equivalence relation (assuming $(\cX,d_\cX)$ is $\delta$-hyperbolic). The {\em Gromov boundary} is the space of equivalence classes of Gromov sequences. We denote it by $\partial  \cX$, leaving the metric implicit.  Let $\overline{\cX}$ denote $\cX \cup \partial \cX$. 

The Gromov product extends to $\partial \cX$ as follows. Let $p,z \in \cX$ and $\xi, \eta \in \partial \cX$. Define
$$(\xi|p)_z := \inf \liminf_{i \to \infty} (x_i|p)_z, \quad (\xi|\eta)_z := \inf \liminf_{i \to \infty} (x_i|y_i)_z$$
where the infimums are over all sequences $\{x_i\}_{i=1}^\infty \in \xi, \{y_i\}_{i=1}^\infty\in \eta$. By \cite[ Lemma 5.11]{Va05}
\begin{eqnarray}\label{eqn:gromovproduct}
\limsup_{i \to \infty} (x_i|y_i)_z -2\delta \le (\xi|\eta)_z \le \liminf_{i \to \infty} (x_i|y_i)_z
\end{eqnarray}
for any sequences $\{x_i\}_{i=1}^\infty \in \xi, \{y_i\}_{i=1}^\infty\in \eta$. These inequalities also hold if $\eta = p \in \cX$ and $y_i$ is any sequence with $\lim_{i\to\infty} y_i=p$. According to \cite[Proposition 5.12]{Va05},  inequality (\ref{eqn:gromov}) extends to $x,y \in \partial \cX$.

In \cite{BH99} it is shown that if $\epsilon>0$ is sufficiently small and $\bar{d}_\epsilon:\overline{\cX} \times \overline{\cX} \to \RR$ is defined by
$$\bar{d}_\epsilon(\xi,\eta):=e^{-\epsilon (\xi|\eta)_z}$$
then there exists a metric $\bar{d}$ on $\overline{\cX}$ and constants $A,B>0$ such that $A\bar{d}_\epsilon \le \bar{d} \le B\bar{d}_\epsilon.$ Any such metric is called a {\em visual metric}. The topology on $\cX$ induced by $d_\cX$ agrees with the topology induced by $\bar{d}$. Moreover $\cX$ is dense in $\bar{\cX}$.

\subsection{Quasi-conformal measures and horofunctions}

Let $(\cX,d_\cX)$ be a $\delta$-hyperbolic metric space. Choose a basepoint $x_0\in \cX$. 
\begin{lem}
Let $\xi \in \partial \cX$ and suppose $\{y_i\}, \{z_i\} \subset \cX$ are two sequences converging to $\xi$ (w.r.t. the topology on $\overline{\cX}$). Then for any $w \in \cX$,
$$\limsup_{i\to\infty} \left| d_\cX(y_i,w) - d_\cX(y_i,x_0) - \Big( d_\cX(z_i,w) - d_\cX(z_i,x_0) \Big) \right| \le 4\delta.$$
\end{lem}
\begin{proof}
Observe that
$$d_\cX(y_i,w) - d_X(y_i,x_0)  = d_\cX(w,x_0) - 2(y_i| w)_{x_0}.$$
A similar statement holds for $z_i$ in place of $y_i$. Thus,
$$\left| d_\cX(y_i,w) - d_\cX(y_i,x_0) - \Big( d_\cX(z_i,w) - d_\cX(z_i,x_0) \Big) \right|= 2\left| (y_i|w)_{x_0} - (z_i|w)_{x_0} \right|.$$
The lemma now follows from (\ref{eqn:gromovproduct}).
\end{proof}
For $\xi \in \partial \cX$, define $h_\xi:\cX \to \RR$ by
$$h_\xi(z): = \inf \liminf_{n\to\infty} d_\cX(z,y_i) - d_\cX(y_i,x_0)$$
where the infimum is over all sequences $\{y_i\} \subset \cX$ which converge to $\xi$. This is the {\em horofunction} associated to $\xi$ (and the basepoint $x_0$). By the previous lemma, if $\{x_i\}$ is any sequence converging to $\xi$ and $z \in \cX$ is arbitrary then
\begin{eqnarray}\label{eqn:horo}
\limsup_{i\to\infty} \left|h_\xi(z) - \Big(d_\cX(x_i,z) - d_\cX(x_i,x_0)\Big) \right| \le 4 \delta.
\end{eqnarray}

\begin{defn}[Quasi-conformal measure]\label{defn:qc}
Suppose $(\Gamma,d)$ is a Gromov hyperbolic group. A Borel probability measure $\nu$ on $\partial \Gamma$ is {\em quasi-conformal} if there are constants $\fh,C>0$ such that for any $g\in \Gamma$ and a.e. $\xi \in \partial \Gamma$,
$$C^{-1} \exp(-\fh h_\xi(g^{-1})) \le \frac{d\nu  \circ g}{d\nu}(\xi) \le C \exp(-\fh h_\xi(g^{-1})).$$
We will call $\fh>0$ the {\em quasi-conformal constant} associated to $\nu$.
\end{defn}

It is well-known that if $d$ comes from a word metric on $\Gamma$ (or more generally, any geodesic metric) then there is a quasi-conformal measure on $\partial \Gamma$ \cite{Co93}. More generally:

\begin{lem}\label{lem:conformal}
Let $(\Gamma,d)$ be a non-elementary, uniformly quasi-geodesic, hyperbolic group. Then there exists a quasi-conformal measure $\nu$ on $\partial \Gamma$. Moreover, any two quasi-conformal measures are equivalent. Also if $\fh$ is the quasi-conformal constant of $\nu$ then there is a constant $C>0$ such that
\begin{enumerate}
\item If $B(g,r)$ denotes the ball of radius $r$ centered at $g\in \Gamma$ then 
$$C^{-1} e^{ \fh r} \le |B(g,r)| \le C e^{\fh r},\quad \forall g\in \Gamma, r>0.$$

\item $$C^{-1}e^{-\fh n} \le \nu\left(\{ \xi' \in \partial \Gamma:~ (\xi|\xi')_e \ge n \}\right) \le C e^{-\fh n}, \quad \forall n>0, \xi \in \partial \Gamma.$$
\end{enumerate}
\end{lem}
\begin{proof}
This follows immediately from \cite[Theorem 2.3]{BHM11} and the fact that any non-elementary uniformly quasi-geodesic hyperbolic group is a proper quasi-ruled hyperbolic space by Lemma \ref{lem:quasi-ruled} below.
\end{proof}

The paper \cite{BHM11} contains many results for hyperbolic spaces under the assumption that these spaces are {\em quasi-ruled}. To be precise a metric space $(\cX,d_\cX)$ is quasi-ruled if  there are  constants $(\tau,\lambda,c)$ such that $(\cX,d_\cX)$ is $(\lambda,c)$-quasi-geodesic and for any $(\lambda,c)$-quasi-geodesic $\gamma: [a,b] \to \cX$ and any $a \le s \le t \le u \le b$,
$$d_\cX(\gamma(s),\gamma(t)) + d_\cX(\gamma(t),\gamma(u)) - d_\cX(\gamma(s),\gamma(u)) \le 2\tau.$$
\begin{lem}\label{lem:quasi-ruled}
If $(\cX,d_\cX)$ is $(1,c)$-quasi-geodesic then it is quasi-ruled.
\end{lem}
\begin{proof}
If $\gamma:[a,b] \to \cX$ is any $(1,c)$-quasi-geodesic then for any $a \le s \le t \le u \le b$,
$$d_\cX(\gamma(s),\gamma(t)) + d_\cX(\gamma(t),\gamma(u)) - d_\cX(\gamma(s),\gamma(u)) \le 3c.$$
So we may set $\tau=3c/2$.
\end{proof}

The action of $\Gamma$ on its boundary need not be essentially free. For example, consider, if $\Gamma_0$ is word hyperbolic and $F$ is a finite group then $F$, considered as a subgroup of $\Gamma_0\times F$, acts trivially on the Gromov boundary of $\Gamma_0\times F$. However, it does have uniformly bounded stabilizers (in the sense of Definition \ref{defn:ubs}) a condition which is crucial to Theorems \ref{thm:III_1} and \ref{thm:III_lambda}.

\begin{lem}\label{lem:stab}
Let $(\Gamma,d)$ be a non-elementary uniformly quasi-geodesic hyperbolic group and $\nu$ be a quasi-conformal measure on $\partial \Gamma$. Then there is a constant $C>0$ such that $\nu$-a.e. $\xi \in \partial \Gamma$, $|\Stab(\xi)| \le C$.
\end{lem}
\begin{proof}
Let $g \in \Gamma$ have infinite order. It is well-known that there exist distinct elements $\{g^-, g^+\} \subset \partial \Gamma$ such that $\lim_{n\to\infty} g^n  = g^+$ and $\lim_{n\to\infty} g^{-n} = g^-$. Moreover if $\xi \in \partial \Gamma \setminus\{g^-,g^+\}$ then $\lim_{n\to\infty} g^n\xi  = g^+$ and $\lim_{n\to\infty} g^{-n}\xi = g^-$. 

Let $A \subset \partial \Gamma$ be the union of the points $g^-$ and $g^+$ for all infinite-order elements $g\in \Gamma$. Because $\nu$ has no atoms (by Lemma \ref{lem:conformal}) and $A$ is a countable set, $\nu(A)=0$.  Moreover, if $\xi \in \partial \Gamma \setminus A$ and $g$ is any infinite order element then $g \notin \Stab(\xi)$ since $\lim_{n\to\infty} g^n\xi  = g^+ \ne \xi$. 

Thus if $g \in \Stab(\xi)$ and $\xi \in \partial \Gamma \setminus A$ then $g$ has finite order, i.e., $\Stab(\xi)$ is a torsion subgroup. Because the Tits alternative holds for hyperbolic groups \cite{Gr87}, every torsion subgroup of $\Gamma$ is finite.  It is well-known (see e.g., \cite{BG95}, \cite{BH99} or \cite{Br00}) that there is a constant $C>0$ such that for every finite subgroup $H< \Gamma$, $|H| \le C$. This proves $|\Stab(\xi)|\le C$. 
\end{proof}

\subsection{The type of the boundary action}\label{sec:type}

Let $\partial \Gamma$ denote the Gromov boundary of $\Gamma$ and let $\nu$ be a quasi-conformal measure on $\partial \Gamma$ with quasi-conformal constant $\fh$. By \cite{Bo12}, $\Gamma \cc (B,\nu)$ is type $III_\lambda$ for some $\lambda \in (0,1]$. 
\begin{lem}\label{lem:qc-type}
Suppose $\lambda \in (0,1)$. Then there exists a quasi-conformal Borel probability measure $\nu'$ on $B$ such that if
$$R_\lambda(g,\xi)= - \log_\lambda \left( \frac{d \nu' \circ g}{d\nu'}(\xi) \right)$$
then $R_\lambda(g,b) \in \ZZ$ for a.e. $\xi\in \partial \Gamma$.
\end{lem}
\begin{proof}
By \cite{Ad94} $\Gamma \cc (B,\nu)$ is amenable. So by \cite{CFW81}, the orbit-equivalence relation on $B$ is generated by a single measure-class-preserving Borel isomorphism $T:\partial \Gamma \to \partial \Gamma$. If $\lambda \in (0,1)$ then by \cite[Proposition 2.2]{KW91}, there exists a Borel probability measure $\nu'$ on $B$ which is equivalent to $\nu$ such that if
$$R_\lambda(g,\xi)= - \log_\lambda \left( \frac{d \nu' \circ g}{d\nu'}(\xi) \right)$$
then $R_\lambda(g,b) \in \ZZ$ for a.e. $\xi\in \partial \Gamma$ and every $g\in \Gamma$. A careful look at the proof reveals that $\nu'$ can be chosen so that the Radon-Nikodym derivatives between $\nu$ and $\nu'$ are bounded. More precisely, there is a constant $C>0$ such that
$$C^{-1} \le \frac{d\nu'}{d\nu} \le C$$
almost everywhere. Therefore $\nu'$ is also quasi-conformal. 
\end{proof}
We now assume that $\nu=\nu'$ satisfies the lemma above. By quasi-conformality, there exists a constant $C>0$ such that
$$|R_\lambda(g,\xi) -\fh \log_\lambda(e) h_\xi(g^{-1}| \le C.$$
In order to streamline the exposition, let $\L$ denote either $\R$ or $\Z$ depending on whether $\lambda=1$ or $\lambda \in (0,1)$.  Also let
$$R_1(g,\xi)=R(g,\xi) =  \log \left( \frac{d \nu \circ g}{d\nu}(\xi) \right).$$
By quasi-conformality, 
$$|R_1(g,\xi) +\fh  h_\xi(g^{-1})| \le C$$
for some constant $C>0$. So if we let $\fh_\lambda = -\fh \log_\lambda(e)$, then we can say
\begin{eqnarray}\label{eqn:R-h}
|R_\lambda(g,\xi) +\fh_\lambda  h_\xi(g^{-1})| \le C
\end{eqnarray}
for every $\lambda \in (0,1], g\in \Gamma$ and a.e. $\xi \in \partial \Gamma$. 

Next we set some useful notation. Let $\theta_\lambda$ be the measure on $\L$ given by $d\theta_1(t)=e^t dt$ (if $\lambda=1$) and $\theta_\lambda(\{n\}) = \lambda^{-n}$ if $\lambda \in (0,1)$. The {\em Maharam extension} of the action $\Gamma \cc (\partial \Gamma,\nu)$ is the action $\Gamma \cc (\partial \Gamma \times \L, \nu \times \theta_\lambda)$ given by
$$g(\xi,t) = (g\xi, t-R_\lambda(g,\xi)).$$
This action preserves the measure $\nu \times \theta_\lambda$. 

Let $\L^+$ denote the set of positive elements of $\L$. Also, for $A<B \in \L$, we will let $[A,B)_\L$ denote the half-open interval in $\L$ from $A$ to $B$. So if $\L=\Z$, then $[A,B)_\L=\{A,A+1,\ldots,B-1\}$.

\section{Volume growth and regularity}\label{sec:doubling} 

The purpose of this section is to prove Theorem \ref{thm:maximal0} by applying Theorems \ref{thm:III_1}, \ref{thm:III_lambda} and estimating the cardinality of the intersection of balls and horoshells. 
\subsection{Regularity of the averaging sets}

Recall the definition of $R_\lambda(g,b)$ and the Maharam extension $\Gamma \cc \partial\Gamma\times \L$ from the previous section.

\begin{defn}\label{defn:Folner}
Fix $a>0$, $T \in \L^+$ and, for $(\xi,t) \in \partial \Gamma \times [0,T)_\L$, let 
\begin{eqnarray*}
\Gamma_r(\xi,t)&=& \{g\in \Gamma:~d(e,g)-h_\xi(g)-t\le r, g^{-1}(\xi,t) \in \partial \Gamma \times [0,T)_\L\}\\
\cB_r(\xi,t) &=& \{g^{-1}(\xi,t):~g \in \Gamma_r(\xi,t)\}\\
\cS_{a,r}&=&\cB_r \setminus \cB_{r-a}.
\end{eqnarray*}
Although these definitions depend on $T$ we will leave this dependence implicit in the notation.
\end{defn}

We will show that $\cB$ and $\cS_a$ are regular if $a,T$ are sufficiently large. We begin with an estimate of $|\cB_r|$.

\begin{lem}\label{lem:volume2}
There exist constants $a_0,T_0 >0$ such that if $T \ge T_0, a \ge a_0$ then for a.e. $(\xi,t)\in \partial \Gamma \times [0,T)_\L$ 
$$C^{-1}e^{\fh r/2} \le |\cB^{}_r(\xi,t)|, |\Gamma_r(\xi,t)|, |\cS_{r,a}(\xi,t)| \le Ce^{\fh r/2}\quad \forall r \ge 2T+2a$$
where the constant $C>0$ may depend on $T$ but not on $\xi$ or $r$.
\end{lem}

\begin{proof}
Because stabilizers are uniformly bounded  by Lemma \ref{lem:stab}, $C_0^{-1}|\Gamma_r(\xi,t)| \le |\cB_r(\xi,t)| \le C_0 |\Gamma_r(\xi,t)|$ for some $C_0>1$. Because $\cS_{r,a} = \cB_r \setminus \cB_{r-a}$, the bound for $\cB_r$ implies the bound for $\cS_{r,a}$. So it suffices to estimate $|\Gamma_r(\xi,t)|$.

By (\ref{eqn:R-h}) there is a constant $C_1>1$ such that for a.e. $(\xi,t) \in \partial \Gamma \times [0,T)_\L$, if $g\in \Gamma$ is such that $g^{-1}(\xi,t) \in  \partial \Gamma \times [0,T)_\L$, then 
$$ |R_\lambda(g^{-1},\xi)| \le C_1|h_\xi(g)| + C_1, \quad |h_\xi(g)| \le C_1|R_\lambda(g^{-1},\xi)| + C_1.$$
Moreover, $0\le t \le T$ implies $|R_\lambda(g^{-1},\xi)| \le T$. We may assume $T>1$. So,
\begin{eqnarray*}
&&B\left(e,r-T-C_1\right) \cap h^{-1}_\xi\left[ \frac{t-T+C_1}{C_1}, \frac{t-C_1}{C_1} \right] \subset  \Gamma_r(\xi,t) \\
&\subset&  B\left(e,r+(C_1+T)^2\right) \cap h^{-1}_\xi[-C_1-C_1T,C_1+C_1T]
\end{eqnarray*}
where $B(e,r)$ is the ball of radius $r$ centered at the identity in $\Gamma$.

In \cite[Lemma 6.3]{Bo12}, it is shown that there is a constant $T_0>0$ such that if $T_2\ge T_1$ are such that $T_2-T_1\ge T_0$, $\xi\in \partial \Gamma$ and $r\ge \max(|T_1|, |T_2|)-2c$ then
$$ C^{-1} e^{\fh (r+T_2)/2} \le |B(e,r) \cap h^{-1}_\xi[T_1,T_2] |  \le C e^{\fh (r+T_2)/2}$$
where $C>0$ is a constant which may depend on $T_1,T_2$ but not on $r,\xi$. So the inclusions above imply the lemma.
\end{proof}

\begin{lem}
Let $T_0>0$ be as in Lemma \ref{lem:volume2}. If $T>T_0$ then the family $\cB=\{\cB_r\}_{r>0}$ is regular.
\end{lem}

\begin{proof}
Fix $k_0, k_1,k_2 \in \partial \Gamma \times [0,T)_\L$ such that $k_1 \in \cB^{}_r(k_0) \cap \cB^{}_r(k_2)$. To make the notation simpler we will write $x \lesssim y$ if $x \le y + C$ where $C$ is a constant that may depend on $T$ and $(\Gamma,d)$ but not on $r,k_0,k_1, k_2$ or $k_3$.  Of course, $x \gtrsim y$ means $y \lesssim x$ and $x \approx y$ means both $x \lesssim y$ and $y \lesssim x$.

Let $g_1 \in \Gamma_r(k_0)$ be such that $g_1^{-1}k_0=k_1$ and let $g_2 \in \Gamma$ be such that $g_2^{-1}g_1 \in \Gamma_r(k_2)$ and $g_1^{-1}g_2k_2=k_1$. Note $d(e,g_1)\lesssim r$ and $d(g_2^{-1}g_1,e) = d(g_1,g_2) \lesssim r$. 

Let $k_0=(\xi,t)$. Because $\nu$ is quasi-conformal and $g^{-1}_1k_0=k_1 \in \partial \Gamma \times [0,T)_\L$, we must have $|h_\xi(g_1)| \lesssim \fh^{-1}T$ which implies $h_\xi(g_1) \approx 0$. Because $g_1^{-1}g_2 k_2 = k_1 = g_1^{-1}k_0$, we have $k_2=g_2^{-1}k_0$. Therefore, $h_\xi(g_2) \approx 0$ as well.

\noindent {\bf Claim}. $d(g_2,e) \lesssim r$.

\noindent {\em Proof of Claim.} By $\delta$-hyperbolicity,
\begin{eqnarray*}\label{eqn:11}
(e|g_2)_{g_1} \gtrsim \min\{ (e|\xi)_{g_1}, (\xi|g_2)_{g_1} \}.
\end{eqnarray*}
So either
$$2(e|g_2)_{g_1} = d(e,g_1)+d(g_2,g_1)-d(e,g_2) \gtrsim 2(e|\xi)_{g_1} \approx d(e,g_1) + h_\xi(g_1) \approx d(e,g_1) $$
which implies
$$d(e,g_2) \lesssim d(g_2,g_1) \lesssim r$$
or
$$2(e|g_2)_{g_1} = d(e,g_1)+d(g_2,g_1)-d(e,g_2) \gtrsim 2(\xi|g_2)_{g_1} \approx h_\xi(g_1) + d(g_2,g_1) - h_\xi(g_2) \approx d(g_2,g_1)$$
which implies
$$d(e,g_2) \lesssim d(e,g_1) \lesssim r.$$
This proves the claim.


The claim implies $d(e,g_2) - h_\xi(g_2) - t \lesssim r$. Moreover, $g_2^{-1}k_0 = g_2^{-1}g_1k_1 = k_2$. So $g_2 \in \Gamma_{r+C_0}(k_0)$ (for some constant $C_0>0$ which may depend on $T$ and $(\Gamma,d)$ but not on $r$ or the $k_i$'s). Thus $k_2 \in \cB_{r+C_0}^{}(k_0)$. Because $k_0,k_1,k_2$ are arbitrary, this establishes that for any $k_0\in \partial \Gamma\times [0,T)_\L$,
$$\cB_{r+C_0}^{}(k_0) \supset \bigcup_{s \le r} \cB_s^{-1}\cB_r(k_0).$$
By Lemma \ref{lem:volume2}, there is a constant $C_1>0$ such that $\left|\cB_{r+C_0}^{}(k_0)\right| \le C_1 \left|\cB_{r}^{}(k_0)\right|$. Therefore,
$$\left|\bigcup_{s \le r} \cB_s^{-1}\cB_r(k_0) \right| \le C_1 \left|\cB_{r}^{}(k_0)\right|.$$
Since $C_1$ does not depend on $r$ or $k_0$, $\cB^{}$ is regular.
\end{proof}

\begin{lem}\label{lem:regular-trick}
Let $W$ be a set. Let $\cF=\{\cF_r\}_{r>0}$ be a regular family of subset functions on $W$. Suppose there is a constant $C>0$ and a family $\cG=\{\cG_r\}_{r>0}$ of subset functions on $W$ that satisfies $\cG_r(w) \subset \cF_r(x)$ and $|\cG_r(w)| \ge C |\cF_r(w)|$ for every $w\in W$. Then $\cG$ is regular.
\end{lem}
\begin{proof}
Let $C_\cF$ be a regularity constant for $\cF$. Then for any $r>0$ and $w\in W$,
\begin{eqnarray*}
  | \cup_{s\le r} \cG^{-1}_s\cG^{}_r(w) | \le  | \cup_{s\le r} \cF^{-1}_s\cF^{}_r(w) |  \le C_\cF |\cF^{}_r(w)| \le C^{-1} C_\cF |\cG^{}_r(w)|.
  \end{eqnarray*}
    \end{proof}

\begin{cor}\label{cor:regular}
Let $a_0,T_0>0$ be as in Lemma \ref{lem:volume2}. If $a>a_0$ and $T>T_0$ then the family $\cS_{a}=\{\cS_{r,a}\}_{r>0}$ is regular. Moreover, suppose $\epsilon_0>0$ and $\epsilon:\partial\Gamma\times [0,T)_\L \times [0,\infty) \to [-\epsilon_0,\epsilon_0]$ is any function. Define $\tcB^{}_r(\xi,t):=\cB^{}_{r+\epsilon(\xi,t,r)}(\xi,t)$ and $\tcS^{}_{r,a}(\xi,t):=\tcB^{}_r(\xi,t)\setminus \tcB^{}_{r-a}(\xi,t)$ then  $\tcB^{}:=\{\tcB^{}_{r}\}_{r>0}$ and $\tcS_a^{}:=\{\tcS^{}_{r,a}\}_{r>0}$  are regular.
\end{cor}

\begin{proof}
This follows immediately from Lemmas \ref{lem:volume2} - \ref{lem:regular-trick}.
\end{proof}


\begin{cor}\label{cor:mu}
Let $\psi \in L^q(\partial \Gamma,\nu)$ be a probability distribution  (so $\psi \ge 0$ and $\int \psi~d\nu=1$). For $(\xi,t) \in \partial\Gamma\times [0,T)_\L$, let 
$$\Gamma_{r,a}(\xi,t):=\{g\in \Gamma:~ g^{-1}(\xi,t) \in \cS_{r,a}(\xi,t)\}. $$
Define 
$$\kappa^\psi_{r,a}(g) = \frac{1}{T}\int_0^T \int |\Gamma_{r,a}(\xi,t)|^{-1} 1_{\Gamma_{r,a}(\xi,t)}(g^{-1})\psi(\xi)~ d\nu(\xi) dt.$$
If $a$ is sufficiently large then $\{\kappa^\psi_{r,a}\}_{r>0}$ satisfies the strong $L^p$ maximal inequality for all $p>1$ with $\frac{1}{p}+\frac{1}{q} \le 1$. Moreover, if $\psi \in L^\infty(\partial\Gamma,\nu)$, then $\{\kappa^\psi_{r,a}\}_{r>0}$ satisfies the $L\log L$ maximal inequality.

\end{cor}

\begin{proof}
This follows from the previous corollary, Lemma \ref{lem:stab} and Theorem \ref{thm:III_1}. 
\end{proof}

\subsection{Bounding the ball averages}

We now turn to show that  when choosing $\psi=1$, the integral $\kappa_{r,a}^1$ of the averaging sets (defined above) dominates the uniform measure $\beta_r$ on the group.  
A preliminary step is to show that $\kappa_{r,a}^1$ dominates the measure uniformly distributed on a spherical shell $S_{r,a}=S_{r,a}(e)$. 

To that end, for each  $r, b>0$, let $\zeta_{r,b}$ denote the probability measure on $\Gamma$ which is distributed uniformly on the  spherical shell  $B(e,r-a/2+b/2)\setminus B(e,r-a/2-b/2)$. Then the following estimate holds. 

\begin{prop}\label{prop:maximalkey}
For $b$ sufficiently large (depending on $(\Gamma,d)$) and for $a, T$  sufficiently large (depending on $b$ and $(\Gamma,d)$) there is a constant $C>0$ (which may depend on $a,b,T$) such that $\zeta_{r,b} \le C \kappa^1_{r,a}$ for all $r>0$.
\end{prop}

To prove Proposition \ref{prop:maximalkey}, we need the following geometric lemmas.

\begin{lem}\label{lem:existence}
There is a constant $C>0$ such that for any $g\in \Gamma$ there exists $\eta \in \partial \Gamma$ with $|h_\eta(g)| \le C$.
\end{lem}

\begin{proof}
\noindent {\bf Claim 1}. There is a constant $C_0>0$ such that for any $g\in \Gamma$ and $r>0$ 
$$\nu(\{\xi \in \partial \Gamma:~ (\xi|g)_e > r\}) \le C_0 e^{-\fh r}.$$
\begin{proof}[Proof of Claim 1]
If $\xi, \eta \in \partial \Gamma$ satisfy $(\xi|g)_e, (\eta|g)_e > r$ then by (\ref{eqn:gromov}),
$$(\xi|\eta)_e \ge \min\{ (\xi|g)_e, (\eta|g)_e\} - \delta > r-\delta.$$
The claim now follow from Lemma \ref{lem:conformal}.
\end{proof}
\noindent {\bf Claim 2}. There is a constant $R>0$ such that for any $x,y,z \in \Gamma$ there exists $\eta \in \partial \Gamma$ such that $(\eta|x)_z, (\eta|y)_z \le R$.
\begin{proof}[Proof of Claim 2]
It follows from Claim 1 that there is an $R>0$ (independent of $x,y,z$) such that 
$$\nu(\{\xi \in \partial \Gamma:~ (\xi|z^{-1}x)_e > R\}) + \nu(\{\xi \in \partial \Gamma:~ (\xi|z^{-1}y)_e > R\}) < 1.$$
Therefore, there is an $\xi \in \partial \Gamma$ such that $(\xi|z^{-1}x)_e, (\xi|z^{-1}y)_e \le R$. But 
$$(\xi|z^{-1}x)_e=(z\xi|y)_z, (\xi|z^{-1}y)_e = (z\xi|y)_z.$$
So set $\eta=z\xi$.
\end{proof}
Now let $g \in \Gamma$. Because $|h_\xi(g)| \le d(e,g)$ (for any $g \in \Gamma, \xi \in \partial \Gamma$) we may assume without loss of generality, that $d(e,g)>c$. Let $\gamma:I \to \Gamma$ be a $(1,c)$-quasi-geodesic from $e$ to $g$ (where $I=[0,r]$ is some interval in the real line). Let $z  = \gamma(d(e,g)/2)$. By Claim 2, there exist a constant $R>0$ and $\eta \in \partial \Gamma$ such that $(\eta|g)_z, (\eta|e)_z \le R$. Using (\ref{eqn:gromovproduct}, \ref{eqn:horo}) we obtain: \begin{eqnarray*}
2R &\ge& 2(\eta|g)_z = h_\eta(z) + d(g,z) - h_\eta(g) +O(\delta) \\
2R &\ge& 2(\eta|e)_z = h_\eta(z) + d(e,z) +O(\delta).
\end{eqnarray*}
Because $d(g,z)  = d(e,z) +O(c)$,
$$R \ge |2(\eta|g)_z-2(\eta|e)_z| \ge |h_\eta(g)| +O(\delta + c).$$ 
\end{proof}

\begin{lem}\label{lem:T1}
There exists a constant $T_1>0$ such that for any $T>T_1$ and any $g\in \Gamma$, 
$$C^{-1}\exp(-\fh d(e,g)/2) \le \nu(\{ \xi\in \partial \Gamma:~ |h_\xi(g)| \le T\})  \le C\exp(-\fh d(e,g)/2)$$
for some constant $C>0$ that is independent of $g$ (but may depend on $T$).
\end{lem} 

\begin{proof}
By Lemma \ref{lem:existence}, there exist a constant $C_0>0$ and $\xi \in \partial \Gamma$ satisfying $|h_\xi(g)|\le C_0$. Note
$$(\xi|g)_e = (1/2)(d(e,g) - h_\xi(g)) + O(\delta) \le d(e,g)/2 + C_0/2 + O(\delta).$$
Suppose $\eta \in \partial \Gamma$ satisfies $(\xi|\eta)_e > \max\{d(e,g)/2, (\xi|g)_e+\delta\}$. Then 
$$(\eta|g)_e \ge \min\{ (\eta|\xi)_e, (\xi|g)_e\}  - \delta=(\xi|g)_e-\delta.$$
Similarly,
$$(\xi|g)_e \ge \min \{ (\xi|\eta)_e, (\eta|g)_e \}-\delta.$$
Since $(\xi|\eta)_e >(\xi|g)_e +\delta$, we must have $(\xi|g)_e \ge (\eta|g)_e-\delta$. Thus $(\xi|g)_e = (\eta|g)_e +O(\delta)$. Since $(\xi|g)_e = (1/2)(d(e,g) - h_\xi(g))+O(\delta)$ (and a similar formula holds for $\eta$), this implies $h_\xi(g) = h_\eta(g)+O(\delta)$. So if $T>0$ is sufficiently large,
$$\{\eta \in \partial \Gamma:~ (\xi|\eta)_e \ge d(e,g)/2+2T\} \subset \{ \eta \in \partial \Gamma:~ |h_\eta(g)| \le T\}).$$
Lemma \ref{lem:conformal} now implies the lower bound.

To obtain the upper bound, suppose that $\eta,\xi \in \partial \Gamma$ satisfy $|h_\xi(g)|, |h_\eta(g)| \le R$ for some constant $R>0$. Then 
$$(\xi|\eta)_e \ge \min\{(\xi|g)_e, (\eta|g)_e \} -\delta \ge d(e,g)/2 - T_0 + O(\delta)$$
implies that for $R>0$ large enough
$$\{\eta' \in \partial \Gamma:~ (\xi|\eta')_e \ge d(e,g)/2-2R\} \supset \{ \eta' \in \partial \Gamma:~ |h_{\eta'}(g)| \le R\}).$$
Lemma \ref{lem:conformal} now implies the upper bound.
\end{proof}

\begin{proof}[Proof of Proposition \ref{prop:maximalkey}]
Recall that  by Lemma   \ref{lem:conformal}, $C^{-1} e^{ \fh r} \le |B(e,r)| \le C e^{\fh r}$ for all $r>0$. 
Choose $b >  \max\set{ a_0, \frac{1}{\fh}\log C}$, $T_2 \ge \max(T_0,T_1)$, $T \ge 10T_2$ and $a > 2(b+T)$ where $a_0,T_0$ are as in Lemma \ref{lem:volume2} and $T_1$ is as in Lemma \ref{lem:T1}.

By definition, $\zeta_{r,b}$ is uniformly distributed on $B(e,r-a/2+b/2) \setminus B(e,r-a/2-b/2)$.  It follows that 
$$\abs{B(e,r-a/2+b/2)\setminus B(e,r-a/2-b/2)}\ge C^{-1} e^{ \fh (r-a/2+b/2)}- C e^{\fh (r-a/2-b/2)}$$
$$= e^{\fh ( r-a/2)}\left(C^{-1}-C  e^{-2\fh b}\right)\ge C(\fh, b) e^{\fh (r-a/2)}\,\,.$$
Therefore, for each $g\in \Gamma$,
$$\zeta_{r,b}(g) \le |B(e,r-a/2+b/2) \setminus B(e,r-a/2-b/2)|^{-1}\le C(\fh, b)^{-1}e^{a/2} e^{-\fh r} \,,$$
so that indeed  $\zeta_{r,b}(g^{-1}) \le C^\prime \exp(-\fh r)$ for some constant $C^\prime>0$ and all $r $ sufficiently large, provided $a$ and $b$ satisfy the conditions above.

On the other hand, Lemma \ref{lem:volume2} implies $|\Gamma_{r,a}(\xi,t)| \le C \exp(\fh r/2)$ for some $C>0$ (and every $(\xi,t) \in \partial \Gamma \times [0,T)_\L$). So,
\begin{eqnarray*}
\kappa^1_{r,a}(g) &=& \frac{1}{T}\int_0^T \int |\Gamma_{r,a}(\xi,t)|^{-1} 1_{\Gamma_{r,a}(\xi,t)}(g^{-1})~ d\nu(\xi) dt\\
&\ge & C^{-1}\exp(-\fh r/2) \frac{1}{T}\int_0^T  \nu(\{ \xi \in \partial \Gamma:~ g^{-1} \in \Gamma_{r,a}(\xi,t)\}) ~dt.
\end{eqnarray*}
Definition \ref{defn:Folner} implies that $g \in \Gamma_{r,a}(\xi,t)$ if and only if:
$$r-a<d(e,g)-h_\xi(g)-t \le r, \quad t -R_\lambda(g^{-1},\xi) \in [0,T)_\L.$$
Because $\nu$ is quasi-conformal, there is a constant $\rho\ge 1$ such that 
$$|R_\lambda(g^{-1},\xi)| \le \rho |h_\xi(g)| + \rho$$
for every $\xi \in \partial \Gamma, g\in \Gamma$. By choosing $T$ larger if necessary, we may assume $T>2\rho T_2 + 2\rho$.

Now suppose $(\xi,t) \in \partial\Gamma \times [0,T)_\L$, $g \in B(e,r-a/2+b/2) \setminus B(e,r-a/2-b/2)$, $\rho T_2 + \rho \le t < T-\rho T_2 - \rho$ and $|h_\xi(g)| \le T_2$. Then 
$$r-a < r - a/2 - b/2 - T \le r-a/2-b/2 - T_2  - (T-\rho T_2 - \rho) \le d(e,g) - h_\xi(g) - t \le r - a/2 + b/2 + T_2 \le r.$$
Also 
$$|R_\lambda(g^{-1},\xi)| \le \rho T_2 + \rho \Rightarrow t + h_\xi(g) \in [0,T)_\L.$$
So $g \in \Gamma_{r,a}(\xi,t)$. Lemma \ref{lem:T1} now implies
\begin{eqnarray*}
\kappa^1_{r,a}(g) &\ge& C^{-1}\exp(-\fh r/2) \frac{1}{T}\int_0^T  \nu(\{ \xi \in \partial \Gamma:~ g^{-1} \in \Gamma_{r,a}(\xi,t)\}) ~dt\\
&\ge &C^{-1}\exp(-\fh r/2) \left(\frac{T-2\rho T_2-2\rho}{T}\right) \nu(\{ \xi \in \partial\Gamma:~ |h_\xi(g)| \le T_2 \}) \\
&\ge& C^{-3} \exp(-\fh r)
\end{eqnarray*}
for some (possibly larger) constant $C>0$. So $\zeta_{r,b}(g) \le C^4\kappa^1_{r,a}(g)$ as required. 
\end{proof}

\subsection{Proofs of Theorem \ref{thm:maximal0} and Theorem \ref{thm:word-metrics}}

\begin{proof}[Proof of Theorem \ref{thm:maximal0}]

It follows from Proposition \ref{prop:maximalkey} that if $\Gamma \cc (X,m)$ is any probability-measure-preserving action and $f \in L^p(X,m)$ is nonnegative then $\zeta_{r,b}(f) \le C \kappa^1_{r,a}(f)$. By Corollary \ref{cor:mu} there exist constants $C_p>0$ ($p \ge 1$) such that $\| \sM_\zeta[f] \|_p \le C \|\sM_\kappa[f] \|_p \le C C_p \|f\|_p$ if $p>1$ where
$$\sM_\zeta[f]=\sup_{r>0} \zeta_{r,b}(|f|), \quad \sM_\kappa[f] = \sup_{r>0} \kappa_{r,a}^1(|f|).$$
Similarly, if $f\in L\log L(X,m)$ then $\| \sM_\zeta[f] \|_1 \le C C_1 \|f\|_{L\log L}$. 
Now let $\beta_r$ be the probability measure on $\Gamma$ uniformly distributed over the ball of radius $r$ centered at the identity. 
Let $\sM_\beta[f] = \sup_{r>0} \beta_r(|f|)$. Because $\beta_r$ can be represented as a convex linear combination of probability measures $\{\zeta_{t,b}\}_{t>0}$,  it follows that $\sM_\beta[f] \le  \sM_\zeta[f]$. Thus $\sM_\beta[f] \le CC_p\|f\|_p$ if $f\in L^p(X,m)$ and $\sM_\beta[f] \le CC_1 \|f\|_1$ if $f \in L\log L (X,m)$ as claimed.

As to the averages $\sigma_{r,a}$, recall that by Lemma \ref{lem:conformal},
$$C^{-1} e^{ \fh r} \le |B(e,r)| \le C e^{\fh r}\,\,,\, r>0.$$
So
$$\abs{B(e,r+a)\setminus B(e,r-a)}\ge C^{-1} e^{ \fh (r+a)}- C e^{\fh (r-a)}= \left(C^{-1}-C  e^{-2\fh a}\right)e^{\fh ( r+a)}= C(\fh, a) e^{\fh (r+a)}\,\,.$$
Hence choosing $a$ sufficiently large,  $\sigma_{r,a}\le C^\prime(\fh,a) \beta_{r+a}$ for some constant $C^\prime(\fh,a)$ as probability measures on $\Gamma$, and hence the maximal inequalities for $\beta_r$ imply the maximal inequalities for $\sigma_{r,a}$. Since $\mu_{r,a}$ are convex combinations of $\sigma_{r,a}$, the maximal inequalities hold for them as well.
\end{proof}

\begin{proof}[Proof of Theorem \ref{thm:word-metrics}] 
Let us now consider the case of a word metric, and first note that Lemma  \ref{lem:conformal} can then be given a sharper form, namely 
$C^{-1} e^{ \fh r} \le |S_r(e)| \le C e^{\fh r}\,\,,\, r>0$, where $S_r(e)$ is the sphere with radius $r$ and center $e$.
To demonstrate this fact briefly, recall that every hyperbolic group is a strongly Markov group. In particular, associated to the pair $(\Gamma,S)$ there is a finite graph $(V,E)$  with a distinguished vertex, whose edges are labelled by  the generators $s\in S$, such that the set of elements $w\in \Gamma$ with $\abs{w}_S=n$ is in bijective, length-preserving correspondence with the set of directed paths of length $n$  in the graph, whose initial vertex is the distinguished one. It was established in \cite{CF10}  that the estimate $C^{-1} e^{ \fh r} \le |B(e,r)| \le C e^{\fh r}\,\,,\, r>0$ (due to \cite{Co93}) implies that the adjacency matrix of the graph $(V,E)$ has the following special property.  There exists a positive eigenvalue of maximum modulus, and for all eigenvalues of maximum modulus, the algebraic and geometric multiplicities coincide. 
For such a matrix, the number of directed paths of length $n$ starting at the distinguished vertex, which is equal to the number of words of length $n$ in the group, does indeed satisfy  the desired estimate (see \cite[Lemma. 3.1.4]{Ca11} for more information, and also \cite{Bo10}).  

It then follows that $\sigma_n\le C \beta_n$ as probability measures on the group, and so $\{\sigma_n\}_{n=1}^\infty$ satisfies the maximal inequality satisfied by $\{\beta_n\}_{n=1}^\infty$. The same holds of course for $\{\mu_n\}_{n=1}^\infty$. Viewing $L^\infty(X,m)$ as a norm-dense subspace of $L^p(X,m)$,  if $\pi_X(\mu_n) f$ converges almost everywhere for every bounded function $f$,  standard arguments  using the maximal inequality imply that $\pi_X(\mu_n)f$ converges almost surely for every $f\in L^p$, $ 1 < p \le \infty$ and in $L\log L$. Pointwise almost sure convergence for bounded functions has been established recently  in  \cite[Corollary 1]{BKK11} (and under an additional assumption also in  \cite{PS11}). Given pointwise convergence, as well as the maximal inequality, norm convergence in $L^p$, $ 1 \le p < \infty$  and in $L\log L$  is a straightforward consequence of Lebesgue's dominated convergence theorem.  

\end{proof}

\section{Asymptotic invariance}\label{sec:invariance} 
In order to prove Theorem \ref{thm:main}, we assume for the rest of the paper that the horofunction boundary coincides with the Gromov boundary of $(\Gamma,d)$. This means that whenever $\{g_i\}_{i=1}^\infty \subset \Gamma$ converges to a point $\xi \in \partial \Gamma$ then the horofunction
$$h_\xi(g):=\lim_{i\to\infty} d(g_i,g)-d(g_i,e)$$
is well-defined. In particular, it depends on $\{g_i\}_{i=1}^\infty$ only through $\xi$. 

Define $\cB_r^{}$ and $\cS_{r,a}^{}$ as in Definition \ref{defn:Folner} for some $T \in \L^+$ and $a>0$ with $a \ge a_0, T\ge T_0$ (where $a_0,T_0$ are as in Lemma \ref{lem:volume2}). Let $\Gamma \cc (\partial \Gamma \times \L,\nu \times \theta_\lambda)$ be the Maharam extension and let $E$ be the induced equivalence relation on $\partial\Gamma \times [0,T)_\L$ (notational conventions are explained in \S \ref{sec:type}). So $(\xi,t)E(\xi',t') \Leftrightarrow \exists g \in \Gamma$ such that $g(\xi,t)=(\xi',t')$.

Most of the work in proving Theorem \ref{thm:main} boils down to the next result the proof of which is the goal of this section.

\begin{thm}\label{thm:asymptoticinvariance}
Let $a \ge a_0, T \ge T_0$ where $a_0,T_0$ are as in Lemma \ref{lem:volume2}. For every $\epsilon_0,r>0$ and $(\xi,t) \in \partial \Gamma \times [0,T]$ there exists $0\le \epsilon(\xi,t,r)<\epsilon_0$ such that if $\tcB^{}_r(\xi,t):=\cB^{}_{r+\epsilon(\xi,t,r)}(\xi,t)$ and $\tcS^{}_{r,a}:=\tcB^{}_r(\xi,t)\setminus \tcB^{}_{r-a}(\xi,t)$ then  $\tcB^{}:=\{\tcB^{}_{r}\}_{r>0}$ and $\tcS_a^{}:=\{\tcS^{}_{r,a}\}_{r>0}$  are asymptotically invariant.
\end{thm}

\subsection{The leafwise metric on the equivalence classes}
To begin the proof, we need a leafwise metric on $E$: given $(\xi,t), (\xi',t') \in \partial \Gamma \times [0,T)_\L$ with $(\xi,t)E(\xi',t')$, let $d_\Gamma( (\xi,t), (\xi',t'))$ be  the minimum value of $d(g,e)$ over all $g\in \Gamma$ with $g(\xi,t)=(\xi',t')$. Most of the work in showing Theorem \ref{thm:asymptoticinvariance} boils down to the next two propositions.

\begin{prop}\label{prop:key}
For $(\xi,t)\in \partial \Gamma \times [0,T)_\L$ and $r>0$ let $\cN_n(\cB^{}_r(\xi,t))$ denote the radius-$n$ neighborhood of $\cB^{}_r(\xi,t)$ with respect to $d_\Gamma$. Then for any $n>0$,
\begin{eqnarray*}
\limsup_{\delta \to 0^+} \limsup_{r\to\infty}  \frac{ |\cN_n(\cB^{}_r(\xi,t))| }{ |\cB^{}_{r+\delta}(\xi,t)| } \le 1.
\end{eqnarray*}
Similarly, 
\begin{eqnarray*}
\limsup_{\delta \to 0^+} \limsup_{r\to\infty}  \frac{ |\cN_n(\cS^{}_{r,a}(\xi,t))| }{|\cB^{}_{r+\delta}(\xi,t)| - |\cB^{}_{r-a-\delta}(\xi,t)| } \le 1.
\end{eqnarray*}
\end{prop}

\begin{prop}\label{prop:volume}
For any $\epsilon_0,r>0$ and a.e. $(\xi,t)\in \partial \Gamma \times [0,T)_\L$ there exists $0 \le \epsilon(\xi,t,r)<\epsilon_0$ such that 
\begin{eqnarray*}
1&=&\lim_{\delta \to 0^+} \limsup_{r\to\infty} \frac{ |\cB_{r+\epsilon(\xi,t,r)+\delta}^{}(\xi,t) | }{|\cB_{r+\epsilon(\xi,t,r)}^{}(\xi,t) | } =\lim_{\delta \to 0^-} \liminf_{r\to\infty} \frac{ |\cB_{r+\epsilon(\xi,t,r)+\delta}^{}(\xi,t) | }{|\cB_{r+\epsilon(\xi,t,r)}^{}(\xi,t) | }. 
\end{eqnarray*}
 \end{prop}
 
 \begin{lem}
There exists a countable set $\Phi \subset [E]$ such that $\Phi$ generates $E$ and for every $\phi \in \Phi$ there exists an $n=n(\phi)>0$ such that  $d_\Gamma((\xi,t),\phi(\xi,t))\le n$ for a.e. $(\xi,t)$.
\end{lem}
\begin{proof}
For $n>0$ let $G_n =\{((\xi,t),(\xi',t')) \in E:~ d_\Gamma((\xi,t),(\xi',t')) \le n\}$. Because $d$ is locally finite,  $(\partial \Gamma\times [0,T)_\L,G_n)$ is a bounded degree graph. By \cite{KST99}, this implies that the Borel edge-chromatic number of $(\partial \Gamma\times [0,T)_\L,G_n)$ is finite. That is, there exists a Borel map $\Omega_n:G_n \to A_n$ (where $A_n$ is a finite set) such that if $((\xi,t),(\xi',t')), ((\xi',t'), (\xi'',t''))  \in G_n$ and $(\xi,t) \ne (\xi'',t'')$ then $\Omega_n((\xi,t),(\xi',t')) \ne \Omega_n((\xi',t'),(\xi'',t''))$. We can also assume without loss of generality that $\Omega_n((\xi,t),(\xi',t'))=\Omega_n((\xi',t'),(\xi,t))$. 

For each element $a \in A_n$, define $\phi_a:\partial \Gamma\times [0,T)_\L \to \partial \Gamma\times [0,T)_\L$ as follows. If $(\xi,t)\in \partial \Gamma\times [0,T)_\L$ and there is a $(\xi',t')$ such that $\Omega_n((\xi,t),(\xi',t'))=a$ then let $\phi_a(\xi,t):=(\xi',t')$. Otherwise let $\phi_a(\xi,t):=(\xi,t)$. Then $\phi_a \in [E]$. Moreover, if $((\xi,t),(\xi',t'))\in G_n$ then there is some $a \in A_n$ such that $\phi_a(\xi,t)=(\xi',t')$. Since $\cup_{n=1}^\infty G_n = E$, we have that $\Phi:=\cup_{n=1}^\infty \{\phi_a:~ a\in A_n\}$ is generating.
\end{proof}

\begin{proof}[Proof of Theorem \ref{thm:asymptoticinvariance} given Propositions \ref{prop:key}, \ref{prop:volume}]
Let $\Phi$ be the generating set from the previous lemma, $\phi \in \Phi$, $n=n(\phi)$ and $\epsilon(\xi,t,r)$ be as in Proposition \ref{prop:volume}. Then for any $(\xi,t) \in \partial \Gamma\times [0,T)_\L$,
\begin{eqnarray*}
\limsup_{r\to\infty} \frac{|\tcB^{}_r(\xi,t) \vartriangle \phi(\tcB^{}_r(\xi,t))|}{ | \tcB^{}_r(\xi,t)|} &\le & \limsup_{r\to\infty} 2\frac{| \cN_n(\cB^{}_{r+\epsilon(\xi,t,r)}(\xi,t)) \setminus \cB^{}_{r+\epsilon(\xi,t,r)}(\xi,t)|}{ | \cB^{}_{r+\epsilon(\xi,t,r)}(\xi,t)|} \\
&=&  2\left(\limsup_{r\to\infty} \frac{|\cN_n(\cB^{}_{r+\epsilon(\xi,t,r)}(\xi,t))|}{ | \cB^{}_{r+\epsilon(\xi,t,r)}(\xi,t)|}  -1\right)\\
&\le&  2\left(\limsup_{\delta \searrow 0} \limsup_{r\to\infty} \frac{|\cB^{}_{r+\epsilon(\xi,t,r)+\delta}(\xi,t)|}{ | \cB^{}_{r+\epsilon(\xi,t,r)}(\xi,t)|}  -1\right) =0.
\end{eqnarray*}
The last inequality is justified by Proposition \ref{prop:key} and the last equality follows from Proposition \ref{prop:volume}.  Because $\phi \in \Phi$ is arbitrary, this proves $\tcB^{}$ is asymptotically invariant.


Recall that $\tcS^{}_{r,a}(\xi,t) = \tcB^{}_r(\xi,t) \setminus \tcB^{}_{r-a}(\xi,t)$. By Lemma \ref{lem:volume2},
\begin{eqnarray*}
&&\lim_{r\to\infty} \frac{ |\tcS^{}_{r,a}(\xi,t) \vartriangle \phi( \tcS^{}_{r,a}(\xi,t) )|}{|\tcS^{}_{r,a}(\xi,t)|}\\
 &\le & \lim_{r\to\infty} \frac{ |\tcB^{}_r(\xi,t) \vartriangle \phi( \tcB^{}_r(\xi,t) )| +  |\tcB^{}_{r-a}(\xi,t) \vartriangle \phi( \tcB^{}_{r-a}(\xi,t) )|}{|\tcB^{}_r(\xi,t)|} \frac{|\tcB^{}_r(\xi,t)|}{|\tcS^{}_{r,a}(\xi,t)|}\\
&\le& C^2  \lim_{r\to\infty} \frac{ |\tcB^{}_r(\xi,t) \vartriangle \phi( \tcB^{}_r(\xi,t) )| +  |\tcB^{}_{r-a}(\xi,t) \vartriangle \phi( \tcB^{}_{r-a}(\xi,t) )|}{|\tcB^{}_r(\xi,t)|} =0.
\end{eqnarray*}
Since $\phi \in \Phi$ is arbitrary, the implies $\tcS^{}_a$ is asymptotically invariant.
\end{proof}

\subsection{The key geometric argument}
This section proves Proposition \ref{prop:key}. We need a few geometric lemmas to begin.

\begin{lem}
Suppose $\xi_1,\xi_2 \in \partial \Gamma$ and $\xi_1\ne \xi_2$. Then for any $r\in \RR$ there are sets $V_1 \subset \partial \Gamma$, $V_2 \subset \overline{\Gamma}$ $(=\Gamma \cup \partial \Gamma)$ such that
\begin{itemize}
\item $\xi_1 \in V_1, \xi_2 \in V_2$,
\item $V_1$ is open in $\partial \Gamma$, $V_2$ is open in $\overline{\Gamma}$,
\item $V_1 \cap V_2 = \emptyset$,
\item $\forall v_2 \in V_2 \cap \Gamma, ~\forall \eta \in V_1,~ h_\eta(v_2) \ge r.$
\end{itemize}
\end{lem}
\begin{proof}
Let $V_1$ be an open neighborhood of $\xi_1$ whose closure does not contain $\xi_2$. Let $\{W_n\}_{n=1}^\infty$ be any sequence of decreasing open subsets of $\overline{\Gamma}$ such that  $\cap_n W_n = \{\xi_2\}$ and $W_n \cap V_1 = \emptyset$. If the lemma is false then for each $n$ we can find an $x_n \in W_n$ and an $\xi_n \in V_1$ such that $h_{\xi_n}(x_n) < r$. Observe
$$\lim_{n\to\infty} 2(\xi_n|x_n)_e = \lim_{n\to\infty}d(x_n,e)-h_{\xi_n}(x_n) = +\infty.$$
So if $(\xi_\infty,x_\infty)$ is a limit point of $\{(\xi_n,x_n)\}_{n=1}^\infty$ in $\overline{\Gamma} \times \overline{\Gamma}$  then by equation (\ref{eqn:gromovproduct}), $(\xi_\infty|x_\infty)_e=+\infty$. In particular, every limit point of $\{x_n\}_{n=1}^\infty$ is contained in the closure of $V_1$. But the hypotheses on $W_n$ imply $\lim_{n\to\infty} x_n = \xi_2$, a contradiction.
\end{proof}

\begin{lem}
Suppose that $\{\xi_i\}_{i=1}^\infty \subset \partial \Gamma$ and $\lim_{i\to\infty} \xi_i = \xi_\infty$.
 Fix $C \in \RR$ and let $H_i := \{x \in \Gamma:~ h_{\xi_i}(x) \le C\}$ (for $1\le i\le \infty$). If $x_i \in H_i$ for all $i$ then every limit point $y$ of $\{x_i\}_{i=1}^\infty$ in $\overline{\Gamma}$ satisfies $y \in H_\infty \cup \{\xi_\infty\}$. 
 \end{lem}

\begin{proof}
Without loss of generality we may assume $\{x_i\}_{i=1}^\infty$ converges in $\overline{\Gamma}$ to an element $y$. If $y \in \Gamma$ then $x_i=y$ for all $i$ sufficiently large (since $\Gamma$ is locally finite). So $h_{\xi_\infty}(y) = \lim_{i\to\infty}  h_{\xi_i}(x_i) \le C$ and $y \in H_\infty$.

To obtain a contradiction, suppose that $y \in \partial \Gamma$ but $y \ne \xi_\infty$. The previous lemma implies the existence of sets $V_\infty \subset \partial \Gamma$, $V_y \subset \overline{\Gamma}$ such that 
\begin{itemize}
\item $\xi_\infty \in V_\infty, y \in V_y$,
\item $V_\infty$ is open in $\partial \Gamma$, $V_y$ is open in $\overline{\Gamma}$,
\item $V_\infty \cap V_y = \emptyset$,
\item $\forall g \in V_y \cap \Gamma, ~\forall \eta \in V_\infty,~ h_\eta(g) \ge C+1.$
\end{itemize}


For all $n$ sufficiently large, $\xi_n \in V_\infty$. Therefore $H_n$ 
has trivial intersection with $V_y$. Since $x_n \in H_n$, this implies $\lim_{n\to\infty} x_n \notin V_y$. But $\lim_{n\to\infty} x_n = y \in V_y$. This contradiction implies that if $\lim_{n\to\infty} x_n \in \partial \Gamma$ then $\lim_{n\to\infty} x_n = \xi_\infty$ as required.
\end{proof}


\begin{lem}
There exists a function $ \beta=\beta(r,n,t) \ge 0$ such that if $g,g' \in \Gamma$, $\xi\in \partial \Gamma$ and 
$$d(g,e) \ge r, \quad d(g,g')\le n, \quad |h_\xi(g)|\le t$$
then
$$|(d(g,e)-h_\xi(g)) - (d(g',e)-h_\xi(g'))| \le  \beta(r,n,t).$$
Moreover, $\lim_{r\to\infty}  \beta(r,n,t) = 0$ for any $n,t>0$.
\end{lem}

\begin{proof}
The proof is by contradiction. Assuming no such function exists, there are constants $n,t,\epsilon_0>0$, elements $\xi_r  \in \partial \Gamma$ and elements $g_r,g'_r \in \Gamma$ ($\forall r >0$) such that
\begin{itemize}
\item $d(g_r,e) \ge r$, $d(g_r,g'_r) \le n$, $|h_{\xi_r}(g_r)| \le t$,
\item $|(d(g_r,e)-h_{\xi_r}(g_r)) - (d(g'_r,e)-h_{\xi_r}(g'_r))| \ge \epsilon_0.$
\end{itemize}

 After passing to a subsequence if necessary, we may assume that the sequence $\{g_r^{-1}   \xi_r\}_{r=1}^\infty$ converges to an element $\xi^*\in \partial\Gamma$. We claim $\{g_r^{-1}\}_{r=1}^\infty$ also converges to $\xi^*$. To see this, observe that for any $x,g \in \Gamma$ and $\xi \in \partial \Gamma$,
 $$h_{g\xi}(x) = h_{\xi}(g^{-1}x) - h_{\xi}(g^{-1}).$$
 Therefore,
 $$ h_{g_r^{-1}\xi_r}(g_r^{-1}) = h_{\xi_r}(e) - h_{\xi_r}(g_r) \le t~\forall r.$$
Since $\lim_{r\to\infty} d(g_r^{-1},e) = +\infty$ the previous lemma implies $\lim_{r\to\infty} g_r^{-1} = \xi^*$ as claimed.

The claim implies that for any $x\in \Gamma$,
\begin{eqnarray*}
h_{\xi^*}(x) = \lim_{r\to\infty} d(g_r^{-1},x) - d(g_r^{-1},e)= \lim_{r\to\infty} d(e,g_r x) - d(e,g_r).
 \end{eqnarray*}
Since $d(g_r,g'_r) \le n$ for all $r$ and because $(\Gamma,d)$ is locally finite we may assume after passing to a subsequence that there is a $y\in \Gamma$ such that $g_r^{-1}g_r'=y$ for all $r$. By setting $x=y$ in the equation above, we obtain:
\begin{eqnarray*}
\lim_{r\to\infty} d(e,g_r') - d(e,g_r ) &=& h_{\xi^*}(y) = \lim_{r\to\infty}   h_{g_r^{-1}\xi_r}(y)=  \lim_{r\to\infty} h_{\xi_r}(g_r y) - h_{\xi_r}(g_r)\\
&=&  \lim_{r\to\infty} h_{\xi_r}(g_r') - h_{\xi_r}(g_r).
\end{eqnarray*}
This contradicts the assumption $|(d(g_r,e)-h_{\xi_r}(g_r)) - (d(g'_r,e)-h_{\xi_r}(g'_r))| \ge \epsilon_0$.
\end{proof}

\begin{cor}
There is a constant $K>0$ (depending only on $T, (\Gamma,d), \nu)$ such that for any $r,n>0$ and any $(\xi,t) \in \partial \Gamma\times [0,T)_\L$, 
$$\cN_n(\cB^{}_r(\xi,t)) \subset \cB^{}_{r+\beta(r-K,n,K)}(\xi,t).$$
\end{cor}

\begin{proof}
If $(\xi'',t'') \in \cN_n(\cB^{}_r(\xi,t)) \setminus \cB^{}_r(\xi,t)$ then there exists $(\xi',t') \in \cB^{}_r(\xi,t)$ and $g \in \Gamma$ such that $d(e,g) \le n$ and $g(\xi',t')=(\xi'',t'')$. Since $(\xi',t') \in \cB^{}_r(\xi,t)$, there is also a $\gamma \in \Gamma$ such that
$$\gamma^{-1} (\xi,t) =(\xi',t'), \quad d(\gamma,e) - h_\xi(\gamma)-t \le r.$$
Because $t' = t - R_\lambda(\gamma^{-1},\xi)$ and, by (\ref{eqn:R-h}), $|R_\lambda(\gamma^{-1},\xi) + \fh_\lambda h_\xi(\gamma)| < C$ (for some constant $C>0$), we have 
$$|h_\xi(\gamma)| \le \fh_\lambda^{-1}(T+C) \Rightarrow d(\gamma,e) \le r+T+\fh_\lambda^{-1}(T+C).$$
Let $f = \gamma g^{-1}$. So $d(f ,\gamma) =d(g,e) \le n$. Note $f^{-1}(\xi,t) = g\gamma^{-1}(\xi,t) =g(\xi',t')= (\xi'',t'')$. As above, this implies  $|h_\xi(f)| \le \fh_\lambda^{-1}(T+C)$. Since $(\xi'',t'') \notin \cB^{}_r(\xi,t)$, $d(e,f)>r-T-\fh_\lambda^{-1}(T+C)$. 

We now apply the previous lemma to $f$ and $\gamma$ to obtain
 $$\left| d(e,f) - d(e,\gamma) - h_\xi(f) + h_\xi(\gamma)\right| \le \beta(r-K,n,K)$$
 where $K = T+\fh_\lambda^{-1}(T+C)$. Thus
 $$\left| d(e,f)  - h(f) -t \right| \le |d(e,\gamma) -h_\xi(\gamma) -t | + \beta(r-K,n,K) \le r+\beta(r-K,n,K).$$
 This implies $(\xi'',t'') \in  \cB^{}_{r+\beta(r-K,n,K)}(\xi,t)$ as required.
\end{proof}

Proposition \ref{prop:key} follows from the corollary above and the fact that $\lim_{r\to\infty}  \beta(r,n,t) = 0\quad \forall n,t.$

\subsection{Proof of asymptotic invariance}

In this section we prove Proposition \ref{prop:volume} whose statement is recalled below.

\noindent {\bf Proposition \ref{prop:volume}}.
{\it 
For any $\epsilon_0,r>0$ and a.e. $(\xi,t)\in \partial \Gamma\times [0,T)_\L$ there exists $0 \le \epsilon(\xi,t,r)<\epsilon_0$ such that 
\begin{eqnarray*}
1&=&\lim_{\delta \to 0^+} \limsup_{r\to\infty} \frac{ |\cB_{r+\epsilon(\xi,t,r)+\delta}^{}(\xi,t) | }{|\cB_{r+\epsilon(\xi,t,r)}^{}(\xi,t) | } =\lim_{\delta \to 0^-} \liminf_{r\to\infty} \frac{ |\cB_{r+\epsilon(\xi,t,r)+\delta}^{}(\xi,t) | }{|\cB_{r+\epsilon(\xi,t,r)}^{}(\xi,t) | } 
\end{eqnarray*} }


\begin{proof}
Let $(\xi,t)\in \partial \Gamma\times [0,T)_\L$, $0<a < b$, $l=b-a$ and $1\le N,m$ be integers such that $N$ is divisible by $4$. Suppose that for every $c \in [a+2l/N, b-2l/N]$,
$$\frac{ |\cB_{c-2l/N}^{}(\xi,t) | }{|\cB_{c+2l/N}^{}(\xi,t) | }  \le 1-1/m.$$
By Lemma \ref{lem:volume2},
\begin{eqnarray*}
Ce^{\fh b/2}&\ge& |\cB_b^{}(\xi,t)| = |\cB_{a}^{}(\xi,t)| \prod_{j=0}^{(N/4-1)} \frac{ |\cB_{a + (4j+4)l/N}^{}(\xi,t) | }{|\cB_{a+4jl/N}^{}(\xi,t) | } \\
&\ge &  |\cB_{a}^{}(\xi,t)|(1-1/m)^{-N/4} \ge C^{-1}e^{\fh a/2}(1-1/m)^{-N/4+3}.
\end{eqnarray*}
So $N \le \frac{-4\log(C^2 e^{\fh l/2})-12\log(1-1/m)}{\log(1-1/m)}.$

Suppose now that $N >  \frac{-4\log(C^2 e^{\fh l/2})-12\log(1-1/m)}{\log(1-1/m)}$. Then there exists $c \in [a+2l/N, b-2l/N]$ such that
$$\frac{ |\cB_{c-2l/N}^{}(\xi,t) | }{|\cB_{c+2l/N}^{}(\xi,t) | }  \ge 1-1/m.$$
For any  $x \in [c-l/N,c+l/N]$,
$$\frac{ |\cB_{x}^{}(\xi,t) | }{|\cB_{x+l/N}^{}(\xi,t) | }  \ge \frac{ |\cB_{c-2l/N}^{}(\xi,t) | }{|\cB_{c+2l/N}^{}(\xi,t) | }  \ge 1-1/m$$
and
$$\frac{ |\cB_{x-l/N}^{}(\xi,t) | }{|\cB_{x}^{}(\xi,t) | }  \ge \frac{ |\cB_{c-2l/N}^{}(\xi,t) | }{|\cB_{c+2l/N}^{}(\xi,t) | }  \ge 1-1/m.$$


Now let $\epsilon_0>0$. By induction, for every $r>0$ and $j\ge 2$ there exist $c_{r,j}, N_j>0$ such that:
\begin{itemize}
\item for every $x \in [c_{r,j}-1/N_j,c_{r,j}+1/N_j]$
$$\frac{ |\cB_{x}^{}(\xi,t) | }{|\cB_{x+1/N_j}^{}(\xi,t) | }  \ge 1-1/j, \textrm{  and  } \frac{ |\cB_{x-1/N_j}^{}(\xi,t) | }{|\cB_{x}^{}(\xi,t) | }  \ge 1-1/j;$$
\item $[c_{r,j+1}-1/N_{j+1},c_{r,j+1}+1/N_{j+1}] \subset [c_{r,j}-1/N_{j},c_{r,j}+1/N_j] \subset [r,r+\epsilon_0]$.
\end{itemize}

Let $\delta_{r,\xi,t}$ be the only point in the nested intersection $\bigcap_j [c_{r,j}-1/N_j,c_{r,j}+1/N_j]$. Then $\epsilon(r,\xi,t):=\delta_{r,\xi,t}-r$ satisfies the conclusion of the proposition by construction.
\end{proof}

\subsection{Proof of Theorem \ref{thm:main}}\label{sec:main}

In order to apply Theorems  \ref{thm:III_1} and \ref{thm:III_lambda}, we need to know the action $\Gamma \cc (\partial \Gamma,\nu)$ is weakly mixing, as well as its type and stable type. To prove this, we need the existence of a conformal measure on $\partial \Gamma$:

\begin{lem}\label{lem:conformal2}
Let $(\Gamma,d)$ be a non-elementary, uniformly quasi-geodesic, hyperbolic group whose horofunction boundary coincides with its Gromov boundary. Then there exists a conformal measure $\nu_c$ on $\partial \Gamma$. Thus $\frac{d \nu_c \circ g}{d\nu_c}(\xi) = \exp(-\fh h_\xi(g^{-1}))$ for a.e. $\xi$ and every $g\in \Gamma$.
\end{lem}

\begin{proof}
For $s>0$ and $\gamma \in \Gamma$, let 
$$Z_s(\gamma):= \sum_{g\in \Gamma} e^{-s d(\gamma,g)}.$$
By Lemma \ref{lem:conformal}, there exist constants $C_0,a>0$ so that if
$N_k=|\{g \in \Gamma:~ ak \le d(e,g) < a(k+1)\}|$ then 
$$ C_0^{-1} e^{\fh ak} \le N_k \le C_0 e^{\fh ak}.$$ 
So there is a constant $C_1>0$ such that
$$C_1^{-1} \sum_{k=0}^\infty N_k e^{-ska}\le Z_s(\gamma) \le C_1 \sum_{k=0}^\infty N_k e^{-ska}.$$
So if $s > \fh$,
$$\frac{C_0^{-1}C_1^{-1}}{1-e^{a(\fh-s)}} \le Z_s(\gamma) \le  \frac{C_0C_1}{1-e^{a(\fh-s)}}.$$
For $s>\fh$ let
$$m_s := \frac{1}{Z(s)} \sum_{g\in \Gamma} e^{-s d(e,g)}\delta_g$$
where $\delta_g$ is the Dirac measure concentrated on $\{g\} \subset \Gamma$. We consider these as measures on $\overline{\Gamma}=\Gamma \cup \partial \Gamma$. Let $\nu_c$ be any weak* limit of $m_s$ as $s \searrow \fh$. Because $\lim_{s \searrow \fh} Z(s)=+\infty$, $\nu_c$ is supported on $\partial \Gamma$. An exercise left to the reader shows that $\nu_c$ is conformal as claimed.

\end{proof}

\begin{remark}
If $\lambda \in (0,1)$ then we do not know whether there exists a conformal measure $\nu_c$ on $\partial \Gamma$ which in addition satisfies $\log_\lambda\left(\frac{d\nu_c \circ g}{d\nu_c}(\xi)\right) \in \ZZ$ for every $g \in \Gamma$ and a.e. $\xi \in \partial \Gamma$. However, Lemma \ref{lem:qc-type} implies that $\nu_c$ is equivalent to a quasi-conformal measure $\nu$ which satisfies this condition.
\end{remark}

\begin{lem}
Let $(\Gamma,d)$ be a non-elementary uniformly quasi-geodesic  hyperbolic group whose Gromov boundary coincides with its horofunction boundary and $\nu$ be a quasi-conformal measure on $\partial \Gamma$. Then $\Gamma \cc (\partial \Gamma,\nu)$ is weakly mixing, type $III_\lambda$ and stable type $III_\tau$ for some $\lambda,\tau \in (0,1]$.
\end{lem}

\begin{proof}
By Lemma \ref{lem:conformal} any two quasi-conformal measures on $\partial \Gamma$ are absolutely continuous to each other. So by Lemma \ref{lem:conformal2} we may assume $\nu$ is conformal. So the Radon-Nikodym derivatives $\frac{d\nu \circ g}{d\nu}$ are continuous. By  \cite[Corollary 0.2]{CM07}, $\Gamma \cc (\partial \Gamma,\nu)$ is a factor of a Poisson boundary. By \cite{AL05}, the action of $\Gamma$ on any of its Poisson boundaries is weakly mixing. This implies $\Gamma \cc (\partial \Gamma,\nu)$ is weakly mixing. The main theorem of \cite{Bo12} implies $\Gamma \cc (\partial \Gamma,\nu)$ is type $III_\lambda$ and stable type $III_\tau$ for some $\lambda,\tau \in (0,1]$.
\end{proof}

We now combine Theorems \ref{thm:III_1}, \ref{thm:III_lambda}, Corollary \ref{cor:regular} and Theorem \ref{thm:asymptoticinvariance} to prove Theorem \ref{thm:main}.

\begin{proof}[Proof of Theorem \ref{thm:main}]
Let $a,T \ge 0$ be sufficiently large so that the conclusions to Corollary \ref{cor:regular} and Theorem \ref{thm:asymptoticinvariance} hold. Also let $\epsilon_0>0$.

We use notation as in \S \ref{sec:type}. So let $E$ be the equivalence relation on $\partial \Gamma \times [0,T)_\L$ induced by the partial action of $\Gamma$. Let $\epsilon(\xi,t,r)$, $\tcS_a=\{\tcS_{r,a}\}_{r>0}$ be as in Theorem \ref{thm:asymptoticinvariance}. By Corollary \ref{cor:regular} and Theorem \ref{thm:asymptoticinvariance}, $\tcS_a$ is regular and asymptotically invariant.

Suppose that $\Gamma \cc (\partial \Gamma, \nu)$ has type $III_1$ (so $\L=\R, \lambda=1$). Let
$$\zeta_r(g) = T^{-1} \int_0^T \int |\{
w\in \Gamma:~w(\xi,t)\in \cS_{r,a}(\xi,t)\}|^{-1}1_{\cS_{r,a}(\xi,t)}(g^{-1}(\xi,t))~d\nu(\xi) dt.$$
From (\ref{eqn:R-h}) we have that if $(\xi,t), g^{-1}(\xi,t) \in \partial \Gamma \times [0,T)_\L$ then
$$|h_\xi(g)| \le \fh_\lambda^{-1}(T+C)$$
for some constant $C>0$. Since $r-a < d(e,g) - h_\xi(g)-t \le r$, we have
$$r-a-\fh_\lambda^{-1}(T+C)< d(e,g) \le r +\fh_\lambda^{-1}(T+C) + T.$$
So
$$\{g \in \Gamma:~ g^{-1}(\xi,t) \in \cS_{r,a}(\xi,t)\} \subset B(e,r+\rho) \setminus B(e,r-\rho)$$
where $\rho = a +\fh_\lambda^{-1}(T+C) + T$, which implies $\zeta_r$ is supported on $B(e,r+\rho) \setminus B(e,r-\rho)$. 

The previous lemma and Theorem \ref{thm:III_1} imply $\{\zeta_r\}$ is a pointwise ergodic family in $L\log L$. This finishes the type $III_1$ case. The type $III_\lambda$ case ($\lambda \in (0,1)$) is similar, using Theorem \ref{thm:III_lambda} instead of \ref{thm:III_1}.
\end{proof}

\linespread{0}

\end{document}